%% file: principio_variacional_localmente_compacto.tex
\documentclass[a4paper]{article}

\usepackage{ucs}
\usepackage[utf8x]{inputenc}
\usepackage[english]{babel}

\usepackage[none]{hyphenat}

\usepackage{amsmath}
\usepackage{amsthm}
\usepackage{amsfonts, amssymb}

\usepackage{amscd} 
\usepackage[all,cmtip]{xy} 

\usepackage{graphicx}
\usepackage[final]{hyperref}

\input{definitions}

\title{Entropy and Its Variational Principle
       \\ for Locally Compact Metrizable Systems}

\author
{
  André Caldas%
  \footnote{
    Departamento de Matemática - Universidade de Brasília-DF, Brasil.
    Supported by CNPq grant no. 140888/11-0.
  }
  \and
  Mauro Patrão%
  \footnote{
    Departamento de Matemática - Universidade de Brasília-DF, Brasil.
    Supported by CNPq grant no. 310790/09-3.
  }
}

\begin{document}
  \maketitle

  {\input{principio_variacional_localmente_compacto/parte_principal}}

  \bibliography{bibliografia}
\end{document}

%% file: principio_variacional_localmente_compacto/parte_principal.tex

  {\input{principio_variacional_localmente_compacto/abstract}}
  {\input{principio_variacional_localmente_compacto/introducao}}
  {\input{principio_variacional_localmente_compacto/01_preliminares}}
  {\input{principio_variacional_localmente_compacto/02_principio_variacional_topologico}}
  {\input{principio_variacional_localmente_compacto/03_endomorfismos_de_grupos_de_lie}}

%% file: principio_variacional_localmente_compacto/abstract.tex

\begin{abstract}
  For a given topological dynamical system
  $\dynamicalsystem{X}{T}$
  over a compact set $X$ with a metric $d$,
  the \emph{variational principle} states that
  \begin{equation*}
    \ksentropy{T}
    =
    \topologicalentropy{T}
    =
    \bowenentropy{d}{T},
  \end{equation*}
  where
  $\measureentropy{\mu}{T}$ is the Kolmogorov-Sinai entropy,
  with the supremum taken over every $T$-invariant probability measure,
  $\bowenentropy{d}{T}$ is the Bowen entropy,
  and
  $\topologicalentropy{T}$ is the topological entropy as defined by
  Adler, Konheim and McAndrew.
  In \cite{patrao:entropia},
  the concept of topological entropy
  was adapted for the case where $T$ is a proper map
  and $X$ is locally compact separable and metrizable,
  and the variational principle was extended to
  \begin{equation*}
    \ksentropy{T}
    =
    \topologicalentropy{T}
    =
    \min_d
    \bowenentropy{d}{T},
  \end{equation*}
  where the minimum is taken over every distance compatible with the topology of $X$.
  In the present work,
  we dropped the properness assumption,
  extending the above result for any continuous map $T$.

  We also apply our results to extend some previous formulas for the
  topological entropy of continuous endomorphisms of connected Lie
  groups proved in \cite{patrao_caldas:endomorfismos}.
  In particular,
  we prove that any linear transformation $\function{T}{V}{V}$
  over a finite dimensional vector space $V$
  has null topological entropy.
\end{abstract}

%% file: principio_variacional_localmente_compacto/introducao.tex

\section{Introduction}

  In this paper, we extend to continuous maps
  $\dynamicalsystem{X}{T}$
  defined over a metrizable locally compact separable space $X$
  the \emph{variational principle} for entropies.
  We have adapted the classical proof of variational principle due to
  Misiurewicz (see Theorem 8.6 of \cite{walters}).

  The original \emph{topological entropy} was defined by
  Adler, Konheim and McAndrew (see \cite{akm:entropia}),
  by adapting the measure theoretic definition,
  which we call \emph{Kolmogorov-Sinai entropy}.
  In this work,
  this will be called
  the \emph{AKM entropy}.
  By assuming the underlying space is equipped with a distance function,
  Bowen and Dinaburg
  (see \cite{bowen:entropia} and \cite{dinaburg})
  have defined a different concept of entropy,
  which we shall call \emph{Bowen entropy}.
  In the case of a compact metrizable space,
  this concept of entropy coincides with the
  \emph{AKM entropy}.
  In particular,
  in the compact case,
  the value of the \emph{Bowen entropy}
  does not depend on the particularly chosen distance.
  That is,
  for a given topological dynamical system
  $\dynamicalsystem{X}{T}$
  over a compact metric space $(X, d)$,
  Dinaburg and Bowen showed that
  \begin{equation*}
    \ksentropy{T}
    =
    \topologicalentropy{T}
    =
    \bowenentropy{d}{T},
  \end{equation*}
  where the supremum is taken over every $T$-invariant probability measure.
  Probably for this reason,
  Bowen himself also named his entropy
  \emph{``topological entropy''}.
  However,
  for non compact spaces,
  the Bowen concept of entropy gives different values
  for equivalent distances over the same space,
  while the \emph{AKM entropy} would always be infinite.
  Later,
  Patrão
  realized that,
  even though the value of the \emph{Bowen entropy}
  depends on the particular chosen distance function,
  by adapting the \emph{AKM entropy},
  to what we are calling \emph{topological entropy} in this paper,
  a variational principle
  relating the measure theoretic \emph{Kolmogorov-Sinai entropy},
  the \emph{topological entropy}
  and
  the \emph{Bowen entropy}
  could be demonstrated,
  as long as the dynamical system $T$ could be extended
  to a metrizable one point compactification of $X$
  (see \cite{patrao:entropia}).

  Inspired by the work of
  Adler, Konheim and McAndrew in \cite{akm:entropia}
  and the work of Patrão in \cite{patrao:entropia},
  we extend the definition of topological entropy for arbitrary topological spaces.
  Inspired by the work of Bowen in \cite{bowen:entropia},
  we simplify his definition of entropy $\bowenentropy{d}{T}$ for non compact sets.
  We call it \emph{$d$-entropy},
  and denote it with superscript: $\dentropy{d}{T}$.
  Then we show
  for the case where $X$ is locally compact metrizable and separable,
  that
  \begin{equation*}
    \ksentropy{T}
    =
    \topologicalentropy{T}
    =
    \min_d
    \bowenentropy{d}{T}
    =
    \min_d
    \dentropy{d}{T},
  \end{equation*}
  where the minimum is attained for every metric $d$
  that can be extended to the one point compactification of $X$
  (see Theorem \ref{th:principio_variacional}).
  Notice that while the metric $d$ might come from the one point compactification,
  this does not mean that the dynamical system $T$ itself
  needs to have an extension to the one point compactification.
  In fact, $T$ can be any continuous system
  over a locally compact separable metrizable space $X$.
  This is a substantial improvement compared to Patrão's preview result
  in \cite{patrao:entropia}.
  In achieving such a result,
  it was fundamental that we were able to restate the
  different definitions of entropy using a unified approach
  that allowed easier comparisons between them.

  We also apply our results to extend some previous formulas for the
  topological entropy of continuous endomorphisms of connected Lie
  groups proved in \cite{patrao_caldas:endomorfismos}.
  In the case of a connected semisimple Lie group,
  without assuming the endomorphism to be surjective,
  we prove that its topological entropy always vanishes.
  In the case of a connected compact Lie group,
  a linear connected reductive Lie group
  or a connected nilpotent Lie group,
  without assuming the endomorphism to be surjective,
  we prove that the topological entropy
  coincides with the topological entropy of the endomorphism's restriction to the
  maximal connected and compact subgroup of the center.
  In particular,
  Proposition \ref{proposition:nilpotent:toral_component}
  implies that any linear transformation $\function{T}{V}{V}$
  over a finite dimensional vector space $V$
  has null topological entropy.
  This extends the result of
  Proposition 4.2 in \cite{patrao:entropia}.

  The paper is organized in the following way.
  In Section \ref{sec:preliminaries},
  we recall some elementary definitions
  related to the different types of entropies,
  extend those definitions to a broader class of dynamical systems,
  and prove some fundamental facts which are used in the sequel.
  In Section \ref{sec:principle}, we prove our main result.
  And in Section \ref{sec:application},
  we apply the main result in order to determine
  the topological entropy for endomorphisms of
  some classes of Lie groups.

%% file: principio_variacional_localmente_compacto/01_preliminares.tex

\section{Preliminaries}
  \label{sec:preliminaries}

  This section is devoted to recalling some elementary definitions
  related to different types of entropy, and to proving some
  fundamental facts which are used in the sequel.
  We also extend the \emph{topological entropy}
  ---
  originally defined for compact systems
  ---
  to an arbitrary topological dynamical system.

  {\input{principio_variacional_localmente_compacto/01/00_operacoes_com_familias_de_conjuntos}}
  {\input{principio_variacional_localmente_compacto/01/01_compactificacao}}
  {\input{principio_variacional_localmente_compacto/01/02_entropia_com_medida}}
  {\input{principio_variacional_localmente_compacto/01/03_entropia_topologica}}
  {\input{principio_variacional_localmente_compacto/01/04_entropia_de_bowen}}
  {\input{principio_variacional_localmente_compacto/01/05_propriedades_da_entropia_de_bowen}}

%% file: principio_variacional_localmente_compacto/01/00_operacoes_com_familias_de_conjuntos.tex

  A \emph{topological dynamical system}
  ---
  or simply a \emph{dynamical system}
  ---
  $\dynamicalsystem{X}{T}$ is a continuous map $T$
  defined over a topological space $X$.
  A \emph{measurable dynamical system}
  $\dynamicalsystem{X}{T}$ is a measurable map $T$
  defined over a measurable space $X$.
  If we embed $X$ with the Borel $\sigma$-álgebra,
  a \emph{topological dynamical system}
  becomes also a \emph{measurable dynamical system}.

  Recall that a family $\family{A}$ of
  subsets of $X$ is a \emph{cover of $X$},
  or simply a \emph{cover}, when
  \begin{equation*}
    X
    =
    \bigcup_{A \in \family{A}} A.
  \end{equation*}
  If the sets in $\family{A}$ are disjoint, then we say that
  $\family{A}$ is a \emph{partition of $X$}.
  A \emph{subcover} of $\family{A}$ is a family
  $\family{B} \subset \family{A}$ which is itself a cover of $X$.
  If $Y \subset X$ and $\family{A}$ is a cover of $X$, then we
  denote by $Y \cap \family{A}$ the cover of $Y$ given by
  \begin{equation*}
    Y \cap \family{A}
    =
    \setsuchthat{A \cap Y}{A \in \family{A}}.
  \end{equation*}
  We denote by $\covercardinality{\family{A}}$
  the least cardinality amongst the subcovers of $\family{A}$.
  For $Y \subset X$,
  $\covercardinality[Y]{\family{A}}$ is a shorthand for
  $\covercardinality{Y \cap \family{A}}$.

  Given two covers $\family{A}$ and $\family{B}$ of an arbitrary
  set $X$, we say that $\family{A}$
  is \emph{finer} then $\family{B}$ or that
  $\family{A}$ \emph{refines} $\family{B}$
  --- and write $\family{B} \prec \family{A}$ ---
  when every element of $\family{A}$
  is a subset of some element of $\family{B}$.
  We also say that $\family{B}$ is \emph{coarser} then $\family{A}$.
  The relation $\prec$ is a \emph{preorder}, and if we identify the
  \emph{symmetric} covers
  (i.e.: covers $\family{A}$ and $\family{B}$
  such that $\family{A} \prec \family{B}$
  and $\family{B} \prec \family{A}$),
  we have a \emph{lattice}.
  As usual,
  $\family{A} \vee \family{B}$ denotes the representative of the coarsest
  covers of $X$ that refines both $\family{A}$ and $\family{B}$, given by
  \begin{equation*}
    \family{A} \vee \family{B}
    =
    \setsuchthat{A \cap B}{A \in \family{A}, B \in \family{B}, A \cap B \neq \emptyset}.
  \end{equation*}

  Given a dynamical system $\dynamicalsystem{X}{T}$
  and a cover $\family{A}$,
  for each $n \in \naturals$ we define
  \begin{equation*}
    \family{A}^n
    =
    \family{A}
    \vee
    T^{-1}(\family{A})
    \vee \dotsb \vee
    T^{-(n-1)}(\family{A})
  \end{equation*}
  If we want to emphasise the dynamical system $T$,
  we write $\family{A}_T^n$ instead.

%% file: principio_variacional_localmente_compacto/01/01_compactificacao.tex
\subsection{Compactification}

  The variational principle has been demonstrated for compact dynamical systems.
  For our extended version,
  we shall treat the dynamical system
  $\dynamicalsystem{X}{T}$ as a subsystem of a compact metrizable one
  (Definition \ref{def:subsistema}).

  For a topological space $X$ to be contained as a subspace
  in its one point compactification $X \cup \set{\infty}$,
  and in order for this one point compactification to be metrizable,
  it is necessary and sufficient that
  $X$ is a metrizable locally compact separable space.

  \begin{definition}[Subsystem]
    \label{def:subsistema}
    We say that a (topological) dynamical system
    $\dynamicalsystem{X}{T}$
    is a \emph{(topological) subsystem} of
    $\dynamicalsystem{Z}{S}$
    when $X \subset Z$ has the induced topology and
    $T(x) = S(x)$ for every $x \in X$.
    We also say that
    $S$ \emph{extends} $T$ to $Z$.
    If
    $\dynamicalsystem{X}{T}$
    and
    $\dynamicalsystem{Z}{S}$
    are measurable dynamical systems instead
    and $X$ is a measurable subset of $Z$,
    we say that $T$ is a \emph{measurable subsystem} of $S$.
  \end{definition}

  \begin{lemma}
    \label{lemma:partition_restriction}
    Suppose $\dynamicalsystem{X}{T}$ is a measurable subsystem of
    $\dynamicalsystem{Z}{S}$.
    If $\family{Z}$ is a covering of $Z$ and
    $\family{C} = X \cap \family{Z}$,
    then
    \begin{equation*}
      \family{C}_T^n
      =
      X
      \cap
      \family{Z}_S^n.
    \end{equation*}
  \end{lemma}

  \begin{proof}
    The case $n = 1$ is trivial.
    Assume it is true for $n$.
    Since
    $T^{-n}(\family{C}) = X \cap S^{-n}(\family{Z})$,
    \begin{align*}
      \family{C}_T^{n+1}
      &=
      \family{C}_T^n
      \cap
      T^{-n}(\family{C})
      \\
      &=
      \left(
        X \cap \family{Z}_S^n
      \right)
      \cap
      \left(
        X \cap S^{-n}(\family{Z})
      \right)
      \\
      &=
      X
      \cap
      \family{Z}_S^{n+1}.
    \end{align*}
  \end{proof}

  \begin{lemma}
    \label{lemma:compactificacao}
    Whenever $X$ is a topological space
    with metrizable one point compactification $X^* = X \cup \set{\infty}$,
    any topological dynamical system $\dynamicalsystem{X}{T}$
    is the subsystem of a dynamical system $\dynamicalsystem{Z}{S}$,
    with $Z$ compact metrizable.
    In this case,
    the natural projection
    \begin{equation*}
      \functionarray{\pi}{Z}{X^*}{x}
                    {\pi(x) = \begin{cases}x, &x \in X\\\infty, &x \not \in X\end{cases}}
    \end{equation*}
    is continuous.
  \end{lemma}

  \begin{proof}
    Let
    \begin{equation*}
      Z
      =
      \prod_{n=0}^{\infty} X^*.
    \end{equation*}
    As a denumerable product of compact metrizable spaces,
    $Z$ is compact metrizable.
    Identify $X$ with a subset of $Z$ by the injection
    \begin{equation*}
      \functionarray{\iota}{X}{Z}{x}{(T^nx)_{n=0}^{\infty}}.
    \end{equation*}
    If we let $\function{\pi_n}{Z}{X^*}$ be the projection
    onto the $n$-th coordinate,
    it is easy to see that $\iota$ is continuous,
    since $\pi_n \circ \iota = T^n$ is continuous for every $n \in \naturals$.
    Also,
    restricted to its image, $\iota$ has continuous inverse $\pi_0$.
    Therefore,
    $X$ is homeomorphic to $\iota(X)$.
    And with this identification,
    $T$ is just the restriction to $\iota(X)$
    of the shift
    \begin{equation*}
      \functionarray{S}{Z}{Z}{(x_n)_{n=0}^{\infty}}{(x_{n+1})_{n=0}^{\infty}}.
    \end{equation*}

    To see that $\pi$ is continuous at $Z \setminus \iota(X)$,
    we just need to show that $\pi^{-1}(A)$ is open for every
    open set $A$ with $\infty \in A$.
    But since $A = X^* \setminus K$ for some compact $K \subset X$,
    we have that
    \begin{equation*}
      \pi^{-1}(A)
      =
      \pi^{-1}
      \left(
        X^* \setminus K
      \right)
      =
      Z \setminus \pi^{-1}(K)
      =
      Z \setminus \iota(K),
    \end{equation*}
    which is open in $Z$, because $\iota(K)$ is compact.
  \end{proof}

  This projection $\pi$,
  on Lemma \ref{lemma:compactificacao},
  induces the pseudometric
  \begin{equation*}
    \widetilde{d}(x,y)
    =
    d(\pi(x), \pi(y))
  \end{equation*}
  over $Z$.
  We denote by the same letter $d$ the restriction of this distance to $X$.
  Since $\pi$ is continuous,
  this pseudometric is such that the ``open balls'' are in fact open.
  Let $\complementset{X} = Z \setminus X$.
  Then,
  these balls are also such that they either contain $\complementset{X}$,
  or have empty intersection with it.

  Under the conditions of
  Lemma \ref{lemma:compactificacao},
  since $X$ is locally compact and separable,
  it is $\sigma$-compact.
  That is,
  $X$ can be written as a denumerable union of compact sets.
  Since compact sets of $X$ are compact sets of $Z$,
  it follows that $X \subset Z$ is a Borel set.
  In this case,
  the Borel sets of $X$ are Borel sets of $Z$,
  and we may restrict Borel measures over $Z$
  to the Borel sets of $X$.
  On the other hand,
  if $\mu$ is a Borel measure over $X$,
  we can extend it to $Z$ by declaring $\mu(Z \setminus X) = 0$.
  Or, equivalently, we might define
  $\mu(C) = \mu(C \cap X)$ for every Borel set $C \subset Z$.
  We shall use the same letter to denote the measure over $Z$
  and its restriction to $X$ or any other Borel subset.
  If we want to make the distinction clear,
  we may write $\mu|_X$.
  According to the following lemma,
  when $\dynamicalsystem{X}{T}$ is a subsystem of
  $\dynamicalsystem{Z}{S}$
  and $\mu$ is an $S$-invariant finite measure,
  then $\mu$ is also $T$-invariant.

  \begin{lemma}
    \label{lemma:s_invariant_is_t_invariant}
    When $\dynamicalsystem{X}{T}$ is a measurable subsystem of
    $\dynamicalsystem{Z}{S}$
    and $\mu$ is an $S$-invariant finite measure,
    then $\mu$ is also $T$-invariant.
  \end{lemma}

  \begin{proof}
    Take a measurable set $A \subset X$.
    Since $\mu$ is $S$-invariant,
    \begin{align*}
      \mu(A)
      +
      \mu
      \left(
        \complementset{X}
      \right)
      &=
      \mu
      \left(
        A
        \cup
        \complementset{X}
      \right)
      \\
      &=
      \mu
      \left(
        S^{-1}(A)
        \cup
        S^{-1}(\complementset{X})
      \right)
      \\
      &\leq
      \mu
      \left(
        S^{-1}(A)
        \cup
        \complementset{X}
      \right)
      \\
      &=
      \mu
      \left(
        T^{-1}(A)
        \cup
        \complementset{X}
      \right)
      \\
      &=
      \mu
      \left(
        T^{-1}(A)
      \right)
      +
      \mu
      \left(
        \complementset{X}
      \right).
    \end{align*}
    This means that
    \begin{equation*}
      \mu(A)
      \leq
      \mu
      \left(
        T^{-1}(A)
      \right).
    \end{equation*}
    But,
    on the other hand,
    \begin{equation*}
      \mu
      \left(
        T^{-1}(A)
      \right)
      \leq
      \mu
      \left(
        S^{-1}(A)
      \right)
      =
      \mu(A).
    \end{equation*}
  \end{proof}

  \begin{lemma}
    \label{lemma:medida_no_complemento}
    Supose $\dynamicalsystem{X}{T}$ is a measurable subsystem of
    $\dynamicalsystem{Z}{S}$.
    Whenever $\mu$ is an $S$-invariant finite measure,
    we have that for any Borel set $A \subset X$,
    and any $n \in \naturals$,
    \begin{equation*}
      \mu
      \left(
        \complementset{X} \cap S^{-n}(A)
      \right)
      =
      0.
    \end{equation*}
  \end{lemma}

  \begin{proof}
    Lemma \ref{lemma:s_invariant_is_t_invariant}
    implies that $\mu$ is also $T$-invariant.
    Since
    $T^{-n}(A) = X \cap S^{-n}(A)$,
    \begin{align*}
      \mu
      \left(
        X \cap S^{-n}(A)
      \right)
      &=
      \mu
      \left(
        T^{-n}(A)
      \right)
      \\
      &=
      \mu(A)
      \\
      &=
      \mu
      \left(
        S^{-n}(A)
      \right)
      \\
      &=
      \mu
      \left(
        X \cap S^{-n}(A)
      \right)
      +
      \mu
      \left(
        \complementset{X} \cap S^{-n}(A)
      \right).
    \end{align*}
    And therefore,
    $
      \mu
      \left(
        \complementset{X} \cap S^{-n}(A)
      \right)
      =
      0
    $.
  \end{proof}

%% file: principio_variacional_localmente_compacto/01/02_entropia_com_medida.tex
\subsection{Kolmogorov-Sinai Entropy}

  Consider the finite measure space $\measurespace{X}$ and a
  finite measurable partition $\family{C}$.
  The \emph{partition entropy} of $\family{C}$ is
  \begin{equation*}
    \partitionentropy{\mu}{\family{C}}
    =
    \sum_{C \in \family{C}}
    \mu(C) \log \frac{1}{\mu(C)}.
  \end{equation*}
  For the measurable dynamical system $\dynamicalsystem{X}{T}$,
  if $\mu$ is a $T$-invariant finite measure,
  the \emph{partition entropy of $T$ with respect to $\family{C}$} is
  \begin{equation*}
    \tpartitionentropy{\mu}{\family{C}}{T}
    =
    \lim_{n \rightarrow \infty}
    \frac{1}{n} \partitionentropy{\mu}{\family{C}^n},
  \end{equation*}
  and the \emph{Kolmogorov-Sinai entropy} of $T$ is
  \begin{equation*}
    \measureentropy{\mu}{T}
    =
    \sup_{\substack{\family{C} \text{: finite} \\ \text{measurable partition}}}
    \tpartitionentropy{\mu}{\family{C}}{T}.
  \end{equation*}

  The most basic properties of the partition entropy
  are a consequence of the concavity of the function
  \begin{equation*}
    f(x)
    =
    -x \log x
    =
    x \log \frac{1}{x}.
  \end{equation*}
  The following lemma states properties that are standard for
  probability measures (see \cite{walters}).
  They are easily restated in terms of finite measures,
  but we shall not need them in this general case.

  \begin{lemma}
    \label{lemma:partition_entropy:properties}
    If $\mu$ is a probability measure
    and $\family{C}$ and $\family{D}$ finite measurable partitions,
    then
    \begin{enumerate}
      \item
        $
          \partitionentropy{\mu}{\family{C} \vee \family{D}}
          \leq
          \partitionentropy{\mu}{\family{C}}
          +
          \partitionentropy{\mu}{\family{D}}
        $.

      \item
        If $\mu_n(C) \rightarrow \mu(C)$ for every $C \in \family{C}$,
        then
        $\partitionentropy{\mu_n}{\family{C}} \rightarrow \partitionentropy{\mu}{\family{C}}$.

      \item
        $\partitionentropy{\mu}{\family{C}} \leq \log \covercardinality{\family{C}}$
    \end{enumerate}
  \end{lemma}

  \begin{lemma}
    \label{lemma:partition_entropy:convex_combination}
    If $\mu = \alpha \gamma + \beta \nu$
    is a convex combination of finite measures
    and
    $\family{C}$ is a finite measurable partition,
    then
    \begin{equation*}
      \alpha
      \partitionentropy{\gamma}{\family{C}}
      +
      \beta
      \partitionentropy{\nu}{\family{C}}
      \leq
      \partitionentropy{\mu}{\family{C}}.
    \end{equation*}
  \end{lemma}

  \begin{proof}
    For each $C \in \family{C}$,
    the concavity of $x \log \frac{1}{x}$ implies that
    \begin{equation*}
      \alpha
      \left(
        \gamma(C)
        \log \frac{1}{\gamma(C)}
      \right)
      +
      \beta
      \left(
        \nu(C)
        \log \frac{1}{\nu(C)}
      \right)
      \leq
      \mu(C) \log \frac{1}{\mu(C)}.
    \end{equation*}
    Now,
    one just has to sum up for all $C \in \family{C}$.
  \end{proof}

  \begin{lemma}
    \label{lemma:partition_entropy:subadditivity}
    If $\mu$ is a finite measure over $X$
    with $0 \leq \mu(X) \leq 1$
    and $Y \subset X$ is measurable,
    then
    \begin{equation*}
      \partitionentropy{\mu}{\family{C}}
      \leq
      \partitionentropy{\mu}{Y \cap \family{C}}
      +
      \partitionentropy{\mu}{\complementset{Y} \cap \family{C}}
    \end{equation*}
    for any finite measurable partition $\family{C}$.
  \end{lemma}

  \begin{proof}
    Fix $C \in \family{C}$.
    Let
    $p = \mu(Y \cap C)$
    and
    $q = \mu(\complementset{Y} \cap C)$.
    If $p = 0$,
    then $\mu(C) = q$,
    and
    \begin{align*}
      \mu(C) \log \frac{1}{\mu(C)}
      &=
      q \log \frac{1}{q}
      \\
      &=
      q \log \frac{1}{q}
      +
      p \log \frac{1}{p}.
    \end{align*}
    A similar conclusion follows when $q = 0$.
    Assume $p,q \neq 0$.
    Then,
    since $p,q \leq \mu(C)$,
    \begin{align*}
      \mu(C) \log \frac{1}{\mu(C)}
      &=
      p \log \frac{1}{\mu(C)}
      +
      q \log \frac{1}{\mu(C)}
      \\
      &\leq
      p \log \frac{1}{p}
      +
      q \log \frac{1}{q}.
    \end{align*}
    The lemma follows if we sum up for $C \in \family{C}$.
  \end{proof}

  Notice that our definition of the Kolmogorov-Sinai entropy does
  not assume that $\mu$ is a probability measure.
  This will prove to be useful when $X$ is metrizable locally compact
  (but not necessarily compact).
  In this case, the set of probability measures is
  not compact in the weak-$*$ topology, while the set of finite
  measures with $0 \leq \mu(X) \leq 1$ is indeed compact.

  \begin{lemma}
    \label{le:ksentropy_com_medidas_leq_1}
    Given a measurable dynamical system $\dynamicalsystem{X}{T}$ and a finite
    $T$-invariant measure $\mu$,
    then, for $\alpha \geq 0$,
    \begin{equation*}
      \measureentropy{\alpha \mu}{T}
      =
      \alpha \measureentropy{\mu}{T}.
    \end{equation*}
  \end{lemma}

  \begin{proof}
    We can assume that $\alpha \neq 0$, since
    $\measureentropy{0}{T} = 0$.

    For any measurable finite partition $\family{C}$,
    \begin{align*}
      \frac{1}{n} \sum_{C \in \family{C}^n}
      \alpha \mu(C) \log \frac{1}{\alpha \mu(C)}
      &=
      \frac{\alpha}{n} \sum_{C \in \family{C}^n}
      \mu(C) \log \frac{1}{\alpha \mu(C)}
      \\ &=
      \frac{\alpha}{n}
      \left(
        \sum_{C \in \family{C}^n}
        \mu(C) \log \frac{1}{\mu(C)}
      \right)
      +
      \frac{\alpha}{n} \mu(X) \log \frac{1}{\alpha}.
    \end{align*}
    Now, we just have to take the limit for $n \rightarrow \infty$.
  \end{proof}

  The following lemma will be very important to reduce
  the prove of the variational principle to the compact case.

  \begin{lemma}
    \label{lemma:entropia_na_compactificacao}
    Let
    $\dynamicalsystem{Z}{S}$
    be a dynamical system and
    $\dynamicalsystem{X}{T}$
    a subsystem with $X \subset Z$ measurable.
    If $\mu$ is an $S$-invariant measure,
    and if
    \begin{equation*}
      \family{Z}
      =
      \set{Z_0, \dotsc, Z_k}
    \end{equation*}
    is a measurable partition of $Z$
    such that
    $Z_1, \dotsc, Z_k \subset X$.
    Then,
    $\mu$ is $T$-invariant and
    \begin{equation*}
      \tpartitionentropy{\mu}{\family{Z}}{S}
      \leq
      \tpartitionentropy{\mu}{\family{C}}{T},
    \end{equation*}
    where $\family{C} = X \cap \family{Z}$.
  \end{lemma}

  \begin{proof}
    According to Lemma \ref{lemma:s_invariant_is_t_invariant},
    $\mu$ is in fact $T$-invariant.
    According to Lemma \ref{lemma:partition_restriction},
    \begin{equation*}
      \family{C}_T^n
      =
      X
      \cap
      \family{Z}_S^n.
    \end{equation*}
    Now,
    Lemma \ref{lemma:partition_entropy:subadditivity} implies that
    \begin{align*}
      \partitionentropy{\widetilde{\mu}}{\family{Z}_{S}^n}
      &\leq
      \partitionentropy{\mu}{X \cap \family{Z}_{S}^n}
      +
      \partitionentropy{\mu}{\complementset{X} \cap \family{Z}_{S}^n}
      \\
      &=
      \partitionentropy{\mu}{\family{C}_{T}^n}
      +
      \partitionentropy{\mu}{\complementset{X} \cap \family{Z}_{S}^n}.
    \end{align*}
    Dividing by $n$ and taking the limit as $n \rightarrow \infty$
    will prove our claim as soon as we show that
    the second term on the right side is constant.

    Since elements $C \in \family{Z}^n$ are of the form
    \begin{equation*}
      C
      =
      C_0
      \cap
      S^{-1}(C_1)
      \cap \dotsb \cap
      S^{-n+1}(C_{n-1}),
    \end{equation*}
    with $C_j \in \family{Z}$,
    Lemma \ref{lemma:medida_no_complemento}
    implies that the only $C \in \complementset{X} \cap \family{Z}^n$
    with non null measure is
    \begin{equation*}
      C
      =
      \complementset{X}
      \cap
      \bigcap_{j=0}^{n-1} S^{-1}(Z_0).
    \end{equation*}
    In particular,
    since $\complementset{X} \cap \family{Z}$ is a finite partition of $\complementset{X}$,
    \begin{equation*}
      \mu(C)
      =
      \mu
      \left(
        \complementset{X}
      \right).
    \end{equation*}
    Therefore,
    \begin{equation*}
      \partitionentropy{\mu}{\complementset{X} \cap \family{Z}_{S}^n}
      =
      \mu
      \left(
        \complementset{X}
      \right)
      \log
      \frac{1}
      {
        \mu
        \left(
          \complementset{X}
        \right)
      }
    \end{equation*}
    is constant.
  \end{proof}

  {\input{principio_variacional_localmente_compacto/01/02_01_entropia_condicional}}

%% file: principio_variacional_localmente_compacto/01/02_01_entropia_condicional.tex
\subsubsection{Conditional Entropy}

  Given a probability measure $\mu$
  and two finite measurable partitions
  $\family{C}$ and $\family{D}$,
  the \emph{conditional entropy} is an important tool to relate
  $\partitionentropy{\mu}{\family{C}}$
  and
  $\partitionentropy{\mu}{\family{D}}$.
  For a measurable set $C$ with $\mu(C) > 0$,
  probability $\mu$ conditioned to $C$,
  $\conditionalprobability{\mu}{C}{\cdot}$
  is given by
  \begin{equation*}
    \conditionalprobability{\mu}{C}{B}
    =
    \frac{\mu(B \cap C)}{\mu(C)}.
  \end{equation*}
  For our purposes,
  when $\mu(C) = 0$,
  the conditional probability can be defined arbitrarily.

  \begin{definition}[Conditional Entropy]
    Given a probability measure $\mu$
    and two finite measurable partitions
    $\family{C}$ and $\family{D}$,
    the \emph{conditional entropy}
    is defined as the expected value
    \begin{equation*}
      \conditionalpartitionentropy{\mu}{\family{D}}{\family{C}}
      =
      \sum_{C \in \family{C}}
      \mu(C)
      \partitionentropy{\conditionalprobability{\mu}{C}{\cdot}}{\family{D}}.
    \end{equation*}
  \end{definition}

  Conditional entropy possesses the following properties.

  \begin{lemma}
    \label{lemma:estimation_using_conditional_entropy}
    Let $\dynamicalsystem{X}{T}$ be a measurable dynamical system
    with $T$-invariant probability measure $\mu$.
    If $\family{C}$ and $\family{D}$ are two measurable finite partitions,
    then
    \begin{equation*}
      \tpartitionentropy{\mu}{\family{C}}{T}
      \leq
      \tpartitionentropy{\mu}{\family{D}}{T}
      +
      \conditionalpartitionentropy{\mu}{\family{C}}{\family{D}}.
    \end{equation*}
  \end{lemma}

  \begin{proof}
    This is item $(iv)$ of Theorem 4.12 from \cite{walters}.
  \end{proof}

%% file: principio_variacional_localmente_compacto/01/03_entropia_topologica.tex
\subsection{Topological Entropy}

  A purely topological concept of entropy for a compact system
  introduced by Adler, Konheim and McAndrew in \cite{akm:entropia},
  analogous to the Kolmogorov-Sinai entropy,
  is the \emph{topological entropy}.
  Dinaburg and Goodman proved (see \cite{dinaburg, goodman})
  the \emph{variational principle},
  which states that for metrizable compact systems,
  the \emph{topological entropy} is equal to the supremum of the
  Kolmogorov-Sinai entropies taken over all $T$-invariant probability measures.
  In \cite{patrao:entropia}, Patrão noticed that when the dynamical
  system admitted a one point compactification,
  the variational principle still holds as long as we adapt the
  original definition of topological entropy.
  In the present paper, we provide a definition of
  \emph{topological entropy} which extends the previous definitions
  and allows us to prove the variational principle for
  metrizable separable locally compact systems.

  \begin{definition}[Cover Entropy]
    Given a cover $\family{A}$ of a set $X$,
    the \emph{cover entropy of $\family{A}$} is
    \begin{equation*}
      \coverentropy{\family{A}}
      =
      \log \covercardinality{\family{A}}.
    \end{equation*}
  \end{definition}

  Motivated by the definition presented in \cite{patrao:entropia},
  we shall restrict our attention to open covers of a certain type.

  \begin{definition}[Admissible Cover]
    \label{def:cobertura_admissivel}

    In a topological space $\topologicalspace{X}$, an open cover
    $\family{A}$ is said to be \emph{admissible} when at least
    one of its elements have compact complement.
    If every set has compact complement,
    $\family{A}$ is said to be \emph{strongly admissible},
    or \emph{s-admissible} for short.
  \end{definition}

  \begin{obs}
    \label{obs:cobertura_admissivel_vem_da_compactificacao}

    If $\widetilde{X}$
    is a compactification of $X$,
    then, for any admissible cover $\family{A}$ of $X$,
    there exists an open cover $\widetilde{\family{A}}$
    of $\widetilde{X}$, such that
    $\family{A} = X \cap \widetilde{\family{A}}$.
    In fact,
    we might simply take for $\widetilde{\family{A}}$,
    the family of all open sets
    $\widetilde{A} \subset \widetilde{X}$ such that
    $X \cap \widetilde{A} \in \family{A}$.
    Notice that,
    $\widetilde{\family{A}}$ is sure to cover $\widetilde{X}$
    for the simple fact that
    there is a compact set $K \subset X$ such that
    $X \setminus K \in \family{A}$, and therefore
    $\widetilde{X} \setminus K$ belongs to
    $\widetilde{\family{A}}$.
    Since $K$ is covered by elements in
    $\widetilde{\family{A}}$,
    we have that $\widetilde{\family{A}}$
    is a cover for $\widetilde{X}$.
  \end{obs}

  \begin{definition}[Topological Entropy]
    \label{def:tcoverentropy}
    \label{def:topologicalentropy}

    For a dynamical system $\dynamicalsystem{X}{T}$ and a cover
    $\family{A}$, the
    \emph{topological entropy of $T$ with respect to $\family{A}$} is
    \begin{equation*}
      \tcoverentropy{\family{A}}{T}
      =
      \lim_{n \rightarrow \infty}
      \frac{1}{n} \coverentropy{\family{A}^n}.
    \end{equation*}
    The \emph{topological entropy of $T$} is
    \begin{equation*}
      \topologicalentropy{T}
      =
      \sup_{\family{A} \text{: admissible}}
      \tcoverentropy{\family{A}}{T}.
    \end{equation*}
  \end{definition}
  Throughout this paper, the term \emph{AKM entropy} refers to the original
  definition of entropy given by Adler, Konheim and McAndrew,
  while the term \emph{topological entropy} refers to our modified
  definition.

  \begin{obs}
    Notice that as in the case of the AKM entropy and the
    Kolmogorov-Sinai entropy, the limit in Definition
    \ref{def:topologicalentropy}
    exists thanks to the inequality
    \begin{equation*}
      \covercardinality{\family{A} \vee \family{B}}
      \leq
      \covercardinality{\family{A}}
      \covercardinality{\family{B}}
    \end{equation*}
    (see Theorem 4.10 in \cite{walters}).
  \end{obs}

  We now state some very basic properties satisfied
  by the topological entropy.
  Most of the arguments are consequence
  of the following simple lemma.

  \begin{lemma}
    \label{le:entropia:propriedades_elementares}

    Given a dynamical system $\dynamicalsystem{X}{T}$,
    let $\family{A}$ and $\family{B}$ be covers of $X$
    such that $\family{B} \prec \family{A}$.
    Then, for all $k \in \naturals$ and every subset
    $Y \subset X$,
    \begin{enumerate}
      \item
        \label{it:le:entropia:propriedades_elementares:iteracao_do_refinamento}
        $\family{B}^k \prec \family{A}^k$.

      \item
        \label{it:le:entropia:propriedades_elementares:cardinalidade_da_cobertura}
        $\covercardinality[Y]{\family{B}^k} \leq \covercardinality[Y]{\family{A}^k}$.

      \item
        \label{it:le:entropia:propriedades_elementares:pulling_back}
        $\covercardinality[T^{-1}(Y)]{T^{-1}(\family{A})}
        \leq \covercardinality[Y]{\family{A}}$.

      \item
        \label{it:le:entropia:propriedades_elementares:entropia_relativa_a_cobertura}
        $\tcoverentropy{\family{B}}{T}
        \leq \tcoverentropy{\family{A}}{T}$.
    \end{enumerate}
  \end{lemma}

  When the space is compact,
  it is a simple fact that
  \begin{equation*}
    \topologicalentropy{T^k}
    =
    k \topologicalentropy{T}.
  \end{equation*}
  For the non-compact case,
  only an inequality follows from a similar argument
  (see Remark \ref{obs:sistema_iterado}).

  \begin{lemma}
    \label{lemma:sistema_iterado}
    Consider the dynamical system $\dynamicalsystem{X}{T}$, and let
    $k \in \naturals$. Then,
    \begin{equation*}
      \topologicalentropy{T^k}
      \leq
      k \topologicalentropy{T}.
    \end{equation*}
  \end{lemma}

  \begin{proof}
    Let $\family{A}$ be an admissible cover of $X$.
    Notice that
    $(\family{A}_T^k)_{T^k}^n = \family{A}_T^{kn}$.
    So,
    \begin{equation*}
      k\frac{1}{kn} \coverentropy{\family{A}_T^{kn}}
      =
      \frac{1}{n} \coverentropy{(\family{A}_T^k)_{T^k}^n}.
    \end{equation*}
    Taking the limit for $n \rightarrow \infty$, we have that
    \begin{equation*}
      k \tcoverentropy{\family{A}}{T}
      =
      \tcoverentropy{\family{A}_T^k}{T^k}.
    \end{equation*}
    And since $\family{A} \prec \family{A}_T^k$,
    \begin{equation*}
      \tcoverentropy{\family{A}}{T^k}
      \leq
      \tcoverentropy{\family{A}^k}{T^k}
      =
      k \tcoverentropy{\family{A}}{T}.
    \end{equation*}
    Now, we just have to take the supremum for every admissible cover
    $\family{A}$ to conclude that
    \begin{equation*}
      \topologicalentropy{T^k}
      \leq
      k \topologicalentropy{T}.
    \end{equation*}
  \end{proof}

  \begin{obs}
    \label{obs:sistema_iterado}
    The fact that $\family{A}_T^k$ is not necessarily an
    admissible cover obstructed our way into showing that
    $\topologicalentropy{T^k} = k \topologicalentropy{T}$.
    In the case of Kolmogorov-Sinai entropy for instance,
    the supremum of
    $\tpartitionentropy{\mu}{\family{C}}{T}$
    is taken for every finite measurable partition $\family{C}$.
    In this case, the demonstration of Lemma
    \ref{lemma:sistema_iterado}
    is easily adapted to show that
    $\measureentropy{\mu}{T^k} = k \measureentropy{\mu}{T}$.
    One just has to notice that since $\family{C}_T^k$ is
    a finite measurable partition,
    \begin{equation*}
      k \tpartitionentropy{\mu}{\family{C}}{T}
      =
      \tpartitionentropy{\mu}{\family{C}_T^k}{T^k}
      \leq
      \measureentropy{\mu}{T^k}.
    \end{equation*}
    For locally compact separable metrizable systems,
    the equality
    $\topologicalentropy{T^k} = k \topologicalentropy{T}$
    will follow from the variational principle
    (Corollary
    \ref{corollary:sistema_iterado}).
  \end{obs}

%% file: principio_variacional_localmente_compacto/01/04_entropia_de_bowen.tex
\subsection{Bowen Entropy}

  Bowen introduced in \cite{bowen:entropia} a definition of entropy
  which coincides with AKM's topological entropy when the dynamical
  system $\dynamicalsystem{X}{T}$ is compact metrizable.
  We shall present Bowen's entropy in a different fashion,
  easier to compare to the topological entropy.

  Choose a distance $d$ in $X$ and,
  given $\varepsilon > 0$, denote by
  \begin{equation*}
    \balls{d}{\varepsilon}
    =
    \setsuchthat{\ball[d]{\varepsilon}{x}}{x \in X}
  \end{equation*}
  the family of balls of radius $\varepsilon$.
  Also, for $n \in \naturals$, we define
  \begin{equation*}
    \iteratedmetric{d}{n}(x,y)
    =
    \max_{0 \leq j < n}
    d(T^j x, T^j y).
  \end{equation*}
  In the literature, $(n,\varepsilon)$-spanning sets are usually defined
  as sets $E \subset X$ for which given any $x \in X$,
  there exists $y \in E$
  such that $\iteratedmetric{d}{n}(x,y) \leq \varepsilon$.
  We adopt an equivalent definition, but in terms of covers.
  We do so,
  in a way that is easier to relate the $(n,\varepsilon)$-spanning sets
  with our enhanced definition of topological entropy.

  \begin{definition}[$(n,\varepsilon)$-Spanning Set]
    Let $\dynamicalsystem{X}{T}$ be a dynamical system with a distance $d$.
    For a given $\varepsilon > 0$ and $n \in \naturals$, a set
    $E \subset X$ is a \emph{$(n,\varepsilon)$-spanning set} when
    \begin{equation*}
      X
      =
      \bigcup_{x \in E}
      \ball[\iteratedmetric{d}{n}]{\varepsilon}{x}.
    \end{equation*}
    That is, the family
    $\setsuchthat{\ball[\iteratedmetric{d}{n}]{\varepsilon}{x}}{x \in E}$
    is a cover for $X$.
  \end{definition}

  \begin{definition}[$d$-Entropy]
    \label{def:d_entropia}

    Let $\dynamicalsystem{X}{T}$ be a dynamical system and $d$ a
    distance for $X$.
    Given $\varepsilon > 0$ and $n \in \naturals$, define
    \begin{equation*}
      \epsilondentropy{d}{T}{\varepsilon}
      =
      \lim_{n \rightarrow \infty}
      \frac{1}{n}
       \log \covercardinality{\balls{d_n}{\varepsilon}},
    \end{equation*}
    and
    \begin{equation*}
      \dentropy{d}{T}
      =
      \sup_{\varepsilon > 0}
      \epsilondentropy{d}{T}{\varepsilon}.
    \end{equation*}
    We denote the usual Bowen entropy by
    \begin{equation*}
      \bowenentropy{d}{T}
      =
      \sup_{\substack{\varepsilon > 0 \\ K \text{: compact}}}
      \limsup_{n \rightarrow \infty}
      \frac{1}{n}
      \log \covercardinality[K]{\balls{d_n}{\varepsilon}}.
    \end{equation*}
  \end{definition}

  \begin{obs}
    \label{obs:kbowenentropy:convergence}

    Again, just like in the case of the topological entropy and the
    Kolmogorov-Sinai entropy, the limit in Definition
    \ref{def:d_entropia}
    exists thanks to the inequality
    \begin{equation*}
      \covercardinality{\balls{d_n}{\varepsilon}}
      \leq
      \covercardinality{\balls{d_q}{\varepsilon}}
      \covercardinality{\balls{d_{n-q}}{\varepsilon}}
    \end{equation*}
    for all $q \in \naturals$ such that $0 < q < n$
    (see Theorem 4.10 in \cite{walters}).
  \end{obs}

  An alternative way to characterize the $d$-entropy
  (and Bowen's entropy) is by means of \emph{separated sets}.
  Misiurewicz's proof of the variational principle needs this characterization.

  Let $\dynamicalsystem{X}{T}$ be a dynamical system with distance $d$.
  Given $\varepsilon > 0$ and $n \in \naturals$,
  we say that a set $S \subset X$ is
  \emph{$(n,\varepsilon)$-separated} when for all pairs of distinct
  points $x,y \in S$,
  $\iteratedmetric{d}{n}(x,y) \geq \varepsilon$.
  For a subset $Y \subset X$ and $\varepsilon > 0$,
  let us write
  $\separatedcardinality{n}{\varepsilon}{Y}$ for the supremum amongst
  the cardinalities of $(n,\varepsilon)$-separated subsets of $Y$.

  \begin{lemma}
    \label{le:separado_e_gerador}
    In a dynamical system $\dynamicalsystem{X}{T}$,
    where $\metricspace{X}{d}$ is a metric space,
    given a subset $Y \subset X$ and $\varepsilon > 0$,
    then, for all $n \in \naturals$,
    \begin{equation*}
      \covercardinality[Y]{\balls{\iteratedmetric{d}{n}}{\varepsilon}}
      \leq
      \separatedcardinality{n}{\varepsilon}{Y}
      \leq
      \covercardinality[Y]{\balls{\iteratedmetric{d}{n}}{\frac{\varepsilon}{2}}}.
    \end{equation*}
  \end{lemma}

  \begin{proof}
    The first inequality follows by the following claim
    and by the existence (through Zorn's lema) of maximal
    $(n,\varepsilon)$-separated sets.
    \subproof
    {
      If $E \subset Y$ is a maximal $(n,\varepsilon)$-separated subset
      of $Y$, then
      \begin{equation*}
        \family{B}_E
        =
        \setsuchthat{Y \cap \ball[\iteratedmetric{d}{n}]{\varepsilon}{x}}{x \in E}
      \end{equation*}
      is a cover of $Y$.
    }
    {
      If $\family{B}_E$ is not a cover, then, taking
      $y \in Y \setminus \bigcup_{x \in E} \ball[\iteratedmetric{d}{n}]{\varepsilon}{x}$,
      we have that the set $E \cup \set{y}$ is
      $(n,\varepsilon)$-separated, infringing the maximality of the
      set $E$.
    }

    For the second inequality,
    if $S \subset Y$ is a $(n,\varepsilon)$-separated set, and
    $\family{B} \subset
    \balls{\iteratedmetric{d}{n}}{\frac{\varepsilon}{2}}$
    covering $Y$, then
    for each $s \in S$,
    there exists an $e(s) \in \family{B}$ such that
    $s \in e(s)$.
    This mapping is injective.
    In fact, if $e(s_1) = e(s_2)$, then
    $\iteratedmetric{d}{n}(s_1, s_2) < \varepsilon$.
    And since $S$ is $(n,\varepsilon)$-separated,
    we must have $s_1 = s_2$.
  \end{proof}

  The following proposition characterizes the d-entropy
  in terms of separated sets,
  and also shows that our formulation of Bowen's entropy
  (Definition \ref{def:d_entropia})
  is equivalent to that of Bowen himself.

  \begin{proposition}
    \label{prop:entropia_com_conjuntos_separados}
    For a dynamical system $\dynamicalsystem{X}{T}$,
    where $\metricspace{X}{d}$ is a metric space,
    \begin{align*}
      \dentropy{d}{T}
      &=
      \sup_{\varepsilon > 0}
      \lim_{n \rightarrow \infty}
      \frac{1}{n} \log \separatedcardinality{n}{\varepsilon}{X}
      \\
      \bowenentropy{d}{T}
      &=
      \sup_{\substack{\varepsilon > 0 \\ K \text{: compact}}}
      \lim_{n \rightarrow \infty}
      \frac{1}{n} \log \separatedcardinality{n}{\varepsilon}{K}.
    \end{align*}
  \end{proposition}

  \begin{proof}
    It is immediate from Lemma \ref{le:separado_e_gerador}.
    Just take the $\log$, divide by $n$, take the limit for
    $n \rightarrow \infty$ and the supremum for $\varepsilon > 0$ in
    \begin{equation*}
      \covercardinality{\balls{\iteratedmetric{d}{n}}{\varepsilon}}
      \leq
      \separatedcardinality{n}{\varepsilon}{X}
      \leq
      \covercardinality{\balls{\iteratedmetric{d}{n}}{\frac{\varepsilon}{2}}}.
    \end{equation*}
    And for the Bowen entropy,
    do the same and also take the supremum over compact $K \subset X$ for
    \begin{equation*}
      \covercardinality[K]{\balls{\iteratedmetric{d}{n}}{\varepsilon}}
      \leq
      \separatedcardinality{n}{\varepsilon}{K}
      \leq
      \covercardinality[K]{\balls{\iteratedmetric{d}{n}}{\frac{\varepsilon}{2}}}.
    \end{equation*}
  \end{proof}

  When defining the $d$-entropy,
  we have used the families $\balls{\iteratedmetric{d}{n}}{\varepsilon}$.
  Notice that those families are not the same as
  $\left[\balls{d}{\varepsilon}\right]^n$.
  The following lemma shows that the families
  $\left[\balls{d}{\varepsilon}\right]^n$
  would work as well, making the $d$-entropy and
  the topological entropy much easier to compare.

  \begin{lemma}
    \label{le:xbowen_com_tcoverentropy}
    Let $\dynamicalsystem{X}{T}$ be a dynamical system with distance $d$.
    Then,
    \begin{align*}
      \dentropy{d}{T}
      &=
      \sup_{\varepsilon > 0} \tcoverentropy{\balls{d}{\varepsilon}}{T}
      \\
      \bowenentropy{d}{T}
      &=
      \sup_{\substack{\varepsilon > 0 \\ K \text{: compact}}}
      \limsup_{n \rightarrow \infty}
      \frac{1}{n}
      \log \covercardinality[K]{\balls{d}{\varepsilon}^n}.
    \end{align*}
  \end{lemma}

  \begin{proof}
    It is enough to show that
    \begin{equation*}
      \left[ \balls{d}{\varepsilon} \right]^n
      \prec
      \balls{\iteratedmetric{d}{n}}{\varepsilon}
      \prec
      \left[ \balls{d}{\frac{\varepsilon}{2}} \right]^n.
    \end{equation*}
    Indeed, this would imply that for any $Y \subset X$,
    \begin{equation*}
      \frac{1}{n}
      \log \covercardinality[Y]{\left[ \balls{d}{\varepsilon} \right]^n}
      \leq
      \frac{1}{n}
      \log \covercardinality[Y]{\balls{\iteratedmetric{d}{n}}{\varepsilon}}
      \leq
      \frac{1}{n}
      \log \covercardinality[Y]{\left[ \balls{d}{\frac{\varepsilon}{2}} \right]^n}.
    \end{equation*}
    And therefore,
    the claim for the first equality will follow
    if we make $n \rightarrow \infty$ and take the
    supremum for $\varepsilon > 0$ and $Y = X$.
    For the second equality,
    instead of taking $Y = X$,
    we also have to take the supremum for every
    $Y \subset X$ compact.

    It is immediate that
    \begin{equation*}
      \balls{\iteratedmetric{d}{n}}{\varepsilon}
      \subset
      \left[ \balls{d}{\varepsilon} \right]^n,
    \end{equation*}
    since every ball in the distance $\iteratedmetric{d}{n}$ has the form
    \begin{equation*}
      \ball[d]{\varepsilon}{x}
      \cap \dotsb \cap
      T^{-(n-1)} \ball[d]{\varepsilon}{T^{n-1}x}.
    \end{equation*}
    Therefore,
    $\left[ \balls{d}{\varepsilon} \right]^n
    \prec
    \balls{\iteratedmetric{d}{n}}{\varepsilon}$.
    On the other hand, a set
    $A \in \left[\balls{d}{\frac{\varepsilon}{2}}\right]^n$ has the form
    \begin{equation*}
      A
      =
      \ball[d]{\frac{\varepsilon}{2}}{x_0}
      \cap \dotsb \cap
      T^{-(n-1)} \ball[d]{\frac{\varepsilon}{2}}{x_{n-1}}.
    \end{equation*}
    In particular, if $A \neq \emptyset$,
    by taking $x \in A$,
    we have that for $j = 0, \dotsc, n - 1$,
    $d(T^j x, x_j) < \frac{\varepsilon}{2}$.  So,
    \begin{equation*}
      \ball[d]{\frac{\varepsilon}{2}}{x_j}
      \subset
      \ball[d]{\varepsilon}{T^j x}.
    \end{equation*}
    Therefore,
    \begin{equation*}
      A
      \subset
      \ball[d]{\varepsilon}{x}
      \cap \dotsb \cap
      T^{-(n-1)} \ball[d]{\varepsilon}{T^{n-1} x}
      \in
      \balls{\iteratedmetric{d}{n}}{\varepsilon}.
    \end{equation*}
  \end{proof}

%% file: principio_variacional_localmente_compacto/01/05_propriedades_da_entropia_de_bowen.tex
  The following lemma states the existence of the
  \emph{Lebesgue number} in a form which is easy to apply
  to the construction of refinements for a given cover.

  \begin{lemma}[Lebesgue Number]
    \label{le:numero_de_Lebesgue}
    Suppose $\metricspace{X}{d}$ is a metric space which admits a
    compactification
    $\metricspace{\widetilde{X}}{\widetilde{d}\phantom{.}}$.
    Let $\family{A} = X \cap \widetilde{\family{A}}$,
    where $\widetilde{\family{A}}$ is an open cover of
    $\widetilde{X}$.
    Then, there exists $\varepsilon > 0$ such that
    \begin{equation*}
      \family{A}
      \prec
      \balls{d}{\varepsilon}.
    \end{equation*}
  \end{lemma}

  \begin{proof}
    Let
    \begin{equation*}
      C_\varepsilon
      =
      \setsuchthat{x \in \widetilde{X}}{\exists A \in \widetilde{\family{A}},\, \ball[\widetilde{d}]{\varepsilon}{x} \subset A}
    \end{equation*}
    be the set of all $x \in \widetilde{X}$ such that the balls centered in
    $x$ with radius $\varepsilon$ is contained in some element of
    $\widetilde{\family{A}}$.
    Now, we shall find
    $\varepsilon > 0$ such that $\widetilde{X} = C_\varepsilon$.

    \subproof
    {
      $\widetilde{X}
      =
      \bigcup_{\varepsilon > 0} C_\varepsilon$.
    }
    {
      For each $x \in \widetilde{X}$, there is a
      $A \in \widetilde{\family{A}}$ containing $x$.
      This means that there exists $\varepsilon > 0$ with
      $\ball[\widetilde{d}]{\varepsilon}{x} \subset A$.
    }

    \subproof
    {
      $C_\varepsilon \subset \interior{C_{\frac{\varepsilon}{2}}}$.
    }
    {
      For $x \in C_\varepsilon$, take
      $A \in \family{A}$ such that
      $\ball[\widetilde{d}]{\varepsilon}{x} \subset A$.
      And notice that
      \begin{equation*}
        \ball[\widetilde{d}]{\frac{\varepsilon}{2}}{x}
        \subset
        C_{\frac{\varepsilon}{2}}.
      \end{equation*}
      In fact, for
      $y \in \ball[\widetilde{d}]{\frac{\varepsilon}{2}}{x}$,
      we have that
      $\ball[\widetilde{d}]{\frac{\varepsilon}{2}}{y}
      \subset \ball[\widetilde{d}]{\varepsilon}{x}$.
      That is, $x \in \interior{C_{\frac{\varepsilon}{2}}}$.
    }

    Joining both claims, we have that
    \begin{equation*}
      \widetilde{X}
      =
      \bigcup_{\varepsilon > 0} \interior{C_\varepsilon}.
    \end{equation*}
    Since $\widetilde{X}$ is compact,
    there exists $\varepsilon > 0$ such that
    $\widetilde{X} = C_\varepsilon$.
    This means that
    $\widetilde{\family{A}} \prec \balls{\widetilde{d}}{\varepsilon}$.
    Taking the intersection with $X$,
    \begin{equation*}
      \family{A}
      =
      X \cap \widetilde{\family{A}}
      \prec
      X \cap \balls{\widetilde{d}}{\varepsilon}.
    \end{equation*}
    Now, we just have to observe that since
    \begin{equation*}
      \ball[d]{\varepsilon}{x}
      =
      X \cap \ball[\widetilde{d}]{\varepsilon}{x}
    \end{equation*}
    for any $x \in X$,
    \begin{equation*}
      \balls{d}{\varepsilon}
      \subset
      X \cap \balls{\widetilde{d}}{\varepsilon}.
    \end{equation*}
    Therefore,
    \begin{equation*}
      \family{A}
      \prec
      \balls{d}{\varepsilon}.
    \end{equation*}
  \end{proof}

  Finally, we relate the topological entropy and the
  $d$-entropy.

  \begin{proposition}
    \label{prop:d_entropia_e_compactificacao}

    Let $\dynamicalsystem{X}{T}$ be a dynamical system and
    $\widetilde{X}$ a metrizable compactification of $X$.
    If $d$ is the restriction to $X$ of a distance $\widetilde{d}$ in
    $\widetilde{X}$, then
    \begin{equation*}
      \dentropy{d}{T}
      =
      \sup_{\widetilde{\family{A}}}
      \tcoverentropy{X \cap \widetilde{\family{A}}}{T},
    \end{equation*}
    where the supremum is taken over all open covers
    $\widetilde{\family{A}}$ of $\widetilde{X}$.
  \end{proposition}

  \begin{proof}
    The family
    $\family{D}
    =
    \setsuchthat{\ball[\widetilde{d}]{\varepsilon}{x}}{x \in X}$
    covers $\widetilde{X}$, and is such that
    $\balls{d}{\varepsilon} = X \cap \family{D}$.
    This implies that
    \begin{equation*}
      \dentropy{d}{T}
      \leq
      \sup_{\widetilde{\family{A}} \text{: open}}
      \tcoverentropy{X \cap \widetilde{\family{A}}}{T}.
    \end{equation*}
    Now, Lemma
    \ref{le:numero_de_Lebesgue}
    implies that for any open cover of $\widetilde{X}$,
    $\widetilde{\family{A}}$, there exists
    $\varepsilon > 0$ such that
    \begin{equation*}
      X \cap \widetilde{\family{A}}
      \prec
      \balls{d}{\varepsilon}.
    \end{equation*}
    And this means that
    \begin{equation*}
      \tcoverentropy{X \cap \widetilde{\family{A}}}{T}
      \leq
      \tcoverentropy{\balls{d}{\varepsilon}}{T}.
    \end{equation*}
    Taking the supremum in $\varepsilon$ and applying the Lemma
    \ref{le:xbowen_com_tcoverentropy},
    we conclude that
    \begin{equation*}
      \tcoverentropy{X \cap \widetilde{\family{A}}}{T}
      \leq
      \dentropy{d}{T}.
    \end{equation*}
    Taking the supremum for all open covers
    $\widetilde{\family{A}}$,
    \begin{equation*}
      \sup_{\widetilde{\family{A}} \text{: open}}
      \tcoverentropy{X \cap \widetilde{\family{A}}}{T}
      \leq
      \dentropy{d}{T}.
    \end{equation*}
  \end{proof}

  When the dynamical system is compact, we know that the $d$-entropy
  does not depend on the distance $d$.
  The following corollary to Proposition
  \ref{prop:d_entropia_e_compactificacao}
  extends this result.

  \begin{corollary}
    \label{corollary:entropia_de_bowen_coincide_para_metricas_admissiveis}

    Let $\dynamicalsystem{X}{T}$ be a dynamical system
    with $X$ admitting a metrizable compactification $\widetilde{X}$.
    If $d$ and $c$ are the restriction to $X$ of distancess
    $\widetilde{d}$ and $\widetilde{c}$ in $\widetilde{X}$, then
    \begin{equation*}
      \dentropy{d}{T}
      =
      \dentropy{c}{T}.
    \end{equation*}
  \end{corollary}

  Starting with a technique similar to what we used in the demonstration of
  Lemma \ref{le:numero_de_Lebesgue},
  we relate Bowen's entropy and the topological entropy.

  \begin{proposition}
    \label{prop:topological_and_bowen_entropy}
    Let $\dynamicalsystem{X}{T}$ be a dynamical system
    with a distance $d$.
    Then,
    \begin{equation*}
      \topologicalentropy{T}
      \leq
      \bowenentropy{d}{T}
      \leq
      \dentropy{d}{T}.
   \end{equation*}
  \end{proposition}

  \begin{proof}
    The Bowen entropy will be calculated according to Lemma
    \ref{le:xbowen_com_tcoverentropy}.
    It is evident from the definitions that
    $\bowenentropy{d}{T} \leq \dentropy{d}{T}$.

    Given an admissible cover $\family{A}$,
    take a compact $K \subset X$
    such that $\complementset{K} \in \family{A}$.
    Let
    \begin{equation*}
      C_\varepsilon
      =
      \setsuchthat{x \in K}{\exists A \in \family{A},\, \ball[d]{2\varepsilon}{x} \subset A}
    \end{equation*}
    be the set of all $x \in K$ such that the ball centered at
    $x$ with radius $2\varepsilon$ is subset of some element in
    $\family{A}$.
    Just as in the proof of Lemma
    \ref{le:numero_de_Lebesgue},
    we can take $\varepsilon > 0$ such that
    $K = C_{\varepsilon}$.

    \subproof
    {
      $\family{A} \prec \balls{d}{\varepsilon}$.
    }
    {
      If $B \in \balls{d}{\varepsilon}$ does not intersect $K$,
      then, $B \subset \complementset{K} \in \family{A}$.
      And if $B$ does intersect $K$ then, taking
      $x \in B \cap K$, from our choice of $\varepsilon$,
      there is an $A_x \in \family{A}$ such that
      \begin{equation*}
        B
        \subset
        \ball[d]{2 \varepsilon}{x}
        \subset
        A_x.
      \end{equation*}
      That is,
      $\family{A} \prec \balls{d}{\varepsilon}$.
    }

    Therefore, defining
    \begin{equation*}
      \family{D}
      =
      \balls{d}{\varepsilon}
      \cup
      \set{\complementset{K}},
    \end{equation*}
    it follows that $\family{A} \prec \family{D}$.
    Let us partition $X$ in $K_0, \dotsc, K_{n-1}$ and
    $\widetilde{K}$, with
    \begin{equation*}
      \widetilde{K}
      =
      \complementset{K}
      \cap
      T^{-1}\left( \complementset{K} \right)
      \cap \dotsb \cap
      T^{-(n-1)}\left( \complementset{K} \right)
      \in
      \family{D}^n
    \end{equation*}
    and, for $m = 0, \dotsc, n-1$,
    \begin{equation*}
      K_m
      =
      \complementset{K}
      \cap
      T^{-1}\left( \complementset{K} \right)
      \cap \dotsb \cap
      T^{-(m-1)}\left( \complementset{K} \right)
      \cap
      T^{-m}(K).
    \end{equation*}
    Notice that, in fact,
    $X = \widetilde{K} \cup \bigcup_{m=0}^{n-1} K_m$.
    So,
    \begin{equation*}
      \covercardinality{\family{D}^n}
      \leq
      \covercardinality[\widetilde{K}]{\family{D}^n}
      +
      \sum_{m=0}^{n-1}
      \covercardinality[K_m]{\family{D}^n}.
    \end{equation*}
    Since $\widetilde{K} \in \family{D}^n$,
    \begin{equation*}
      \covercardinality[\widetilde{K}]{\family{D}^n}
      =
      1.
    \end{equation*}
    For $m = 0, \dotsc, n-1$, denote by
    $\family{B}_m$ the family of sets of the form
    \begin{equation*}
      A_0
      \cap
      T^{-1}(A_1)
      \cap \dotsb \cap
      T^{-(n-1)}(A_{n-1}),
    \end{equation*}
    where $A_j = \complementset{K}$ for $j = 0, \dotsc, m-1$,
    and $A_j \in \balls{d}{\varepsilon}$ for $j \geq m$.
    That is,
    \begin{equation*}
      \family{B}_m
      =
      \complementset{K}
      \cap
      T^{-1}\left( \complementset{K} \right)
      \cap \dotsb \cap
      T^{-(m-1)}\left( \complementset{K} \right)
      \cap
      T^{-m}
      \left(
        \balls{d}{\varepsilon}^{n-m}
      \right)
      \subset
      \family{D}^n.
    \end{equation*}
    If
    $\family{F} \subset \balls{d}{\varepsilon}^{n-m}$
    is a cover of $K$, then
    \begin{equation*}
      \complementset{K}
      \cap
      T^{-1}\left( \complementset{K} \right)
      \cap \dotsb \cap
      T^{-(m-1)}\left( \complementset{K} \right)
      \cap
      T^{-m}(\family{F})
      \subset
      \family{B}_m
    \end{equation*}
    is a cover of $K_m$.
    This way,
    \begin{align*}
      \covercardinality[K_m]{\family{D}^n}
      &\leq
      \covercardinality[K_m]{\family{B}_m}
      \\
      &\leq
      \covercardinality[K]
      {
        \balls{d}{\varepsilon}^{n-m}
      }
      \\
      &\leq
      \covercardinality[K]
      {
        \balls{d}{\varepsilon}^n
      }.
    \end{align*}
    And therefore,
    \begin{align*}
      \covercardinality{\family{A}^n}
      &\leq
      \covercardinality{\family{D}^n}
      \\
      &\leq
      \covercardinality[\widetilde{K}]{\family{D}^n}
      +
      \sum_{m=0}^{n-1}
      \covercardinality[K_m]{\family{D}^n}
      \\
      &\leq
      1 +
      \sum_{m=0}^{n-1}
      \covercardinality[K_m]{\family{D}^n}
      \\
      &\leq
      (n+1)
      \covercardinality[K]{\balls{d}{\varepsilon}^n}.
    \end{align*}
    Taking the $\log$, dividing by $n$ and making
    $n \rightarrow \infty$,
    and using Lemma \ref{le:xbowen_com_tcoverentropy},
    we get that
    \begin{align*}
      \tcoverentropy{\family{A}}{T}
      &\leq
      \limsup_{n \rightarrow \infty}
      \frac{1}{n}
      \left(
        \log
        (n+1)
        +
        \log
        \covercardinality[K]{\balls{d}{\varepsilon}^n}
      \right)
      \\
      &=
      \limsup_{n \rightarrow \infty}
      \frac{1}{n}
      \log
      \covercardinality[K]{\balls{d}{\varepsilon}^n}
      \\
      &\leq
      \bowenentropy{d}{T},
    \end{align*}
    since
    $\frac{1}{n} \log (n+1) \rightarrow 0$.
    Taking the supremum in $\family{A}$,
    \begin{equation*}
      \topologicalentropy{T}
      \leq
      \bowenentropy{d}{T}.
    \end{equation*}
  \end{proof}

  \begin{obs}
    \label{obs:topological_and_xbowen_entropy:desigualdade_falsa}
    When the families $\balls{d}{\varepsilon}$
    are admissible covers, Lemma
    \ref{le:xbowen_com_tcoverentropy} warrants that
    \begin{equation*}
      \bowenentropy{d}{T}
      \leq
      \topologicalentropy{T}.
    \end{equation*}
    Nonetheless, even when $d$ is the restriction of a distance given in
    some compactification of the system,
    $\balls{d}{\varepsilon}$ might not be admissible.
    As an example, just take $\balls{d}{\frac{1}{2}}$ in $(0,1)$ with
    the Euclidean distance.
    No set in $\balls{d}{\frac{1}{2}}$ will ever have compact
    complement and therefore,
    $\balls{d}{\frac{1}{2}}$ is not admissible.
  \end{obs}

%% file: principio_variacional_localmente_compacto/02_principio_variacional_topologico.tex
\section{Variational Principle}
  \label{sec:principle}

  The preparations made in Section
  \ref{sec:preliminaries}
  allow us to use Misiurewicz's demonstration
  of the variational principle almost verbatim.
  Misiurewicz's original article is \cite{misiurewicz}.
  A didatic presentation can be found for example in
  \cite{walters}.

  We are concerned about the supremum of the Kolmogorov-Sinai
  entropies over every $T$-invariant Radon probability measure.
  If there is no such a probability measure,
  we agree that
  \begin{equation*}
    \ksentropy{T}
    =
    0.
  \end{equation*}
  According to Lemma
  \ref{le:ksentropy_com_medidas_leq_1},
  this is the same as taking the supremum over all $T$-invariant
  Radon measures $\mu$ with $0 \leq \mu(X) \leq 1$.
  In this case, there is always an invariant measure.
  Namely, $\mu = 0$.

  The following theorem is our main result.
  The proof will be provided after a few auxiliary results.

  \begin{theorem}
    \label{th:principio_variacional}

    Let $\dynamicalsystem{X}{T}$ be a metrizable locally compact separable
    dynamical system.
    Then,
    \begin{equation*}
      \ksentropy{T}
      =
      \topologicalentropy{T}
      =
      \min_{d}
      \bowenentropy{d}{T}.
    \end{equation*}
    The infimum is attained when $d$ is any distance that can be
    extended to the one point compactification of $X$.
  \end{theorem}

  The following lemma is the part of the variational principle
  that is valid for every topological dynamical system.
  The lemma remarkably generalizes Proposition 1.4
  of \cite{handel}, since the only hypothesis we impose to the system
  $\dynamicalsystem{X}{T}$ is that $T$ is continuous.
  Notice how the concept of \emph{admissible cover} allowed us to
  generalize Proposition 1.4 fo \cite{handel} without even
  appealing to advanced results like the
  \emph{Ergodic Decomposition},
  the
  \emph{Shannon-McMillan-Breiman}
  or to the
  \emph{Birkhoff Ergodic Theorem}.

  \begin{lemma}
    \label{lemma:principio_variacional:desigualdade_facil}
    Let $\dynamicalsystem{X}{T}$ be a dynamical system
    and $\mu$ a $T$-invariant Radon probability measure.
    Then,
    \begin{equation*}
      \measureentropy{\mu}{T}
      \leq
      \topologicalentropy{T}.
    \end{equation*}
    If $X$ is metrizable with distance $d$, then
    \begin{equation*}
      \measureentropy{\mu}{T}
      \leq
      \topologicalentropy{T}
      \leq
      \bowenentropy{d}{T}.
    \end{equation*}
  \end{lemma}

  \begin{proof}
    The last statement follows from Proposition
    \ref{prop:topological_and_bowen_entropy}.
    The first part,
    whose details we will present here,
    has actually been demonstrated by Misiurewicz.
    We start by noticing that for every Radon measure,
    the measure of a set $A$ can be approximated from
    inside by a compact $K \subset A$.

    Let $\mu$ be any $T$-invariant Radon probability measure.
    We shall show that for any $\mu$-invariant
    dynamical system $\dynamicalsystem{X}{S}$,
    \begin{equation}
      \label{eq:lemma:principio_variacional:desigualdade_facil:primeira_desigualdade}
      \measureentropy{\mu}{S}
      \leq
      \topologicalentropy{S}
      +
      2
      +
      \log 2.
    \end{equation}
    In particular,
    this is valid for $T^n$ for any $n \in \naturals$.
    And then,
    Lemma \ref{lemma:sistema_iterado}
    implies that
    \begin{align*}
      \measureentropy{\mu}{T}
      &=
      \frac{1}{n}
      \measureentropy{\mu}{T^n}
      \\
      &\leq
      \frac{1}{n}
      \topologicalentropy{T^n}
      +
      \frac{2 + \log 2}{n}
      \\
      &\leq
      \topologicalentropy{T}
      +
      \frac{2 + \log 2}{n}
      \rightarrow
      \topologicalentropy{T}.
    \end{align*}
    And this will finish the demonstration.

    In order to show the validity of
    equation \refeq{eq:lemma:principio_variacional:desigualdade_facil:primeira_desigualdade}
    for any given $\dynamicalsystem{X}{S}$,
    take a finite measurable partition $\family{C}$ such that
    $\measureentropy{\mu}{S} \leq \tpartitionentropy{\mu}{\family{C}}{S} + 1$.
    Choose for each $C \in \family{C}$,
    $D_C \subset C$ such that
    \begin{equation*}
      \mu(C \setminus D_C)
      \leq
      \frac{1}{\covercardinality{\family{C}} \log \covercardinality{\family{C}}}.
    \end{equation*}
    Let
    $D^* = \bigcup_{C \in \family{C}} (C \setminus D_C)$.
    Then,
    \begin{equation*}
      \mu(D^*)
      \leq
      \frac{1}{\log \covercardinality{\family{C}}}.
    \end{equation*}
    Now,
    define a finite measurable partition $\family{D}$
    and a strongly admissible covering $\family{A}$
    by
    \begin{align*}
      \family{D}
      &=
      \setsuchthat{D_C}{C \in \family{C}}
      \cup
      \set{D^*}
      \\
      \family{A}
      &=
      \setsuchthat{D_C \cup D^*}{C \in \family{C}}.
    \end{align*}

    \subproof
    {
      $
        \conditionalpartitionentropy{\mu}{\family{C}}{\family{D}}
        \leq
        1
      $.
    }
    {
      First,
      notice that for every $C \in \family{C}$,
      $\conditionalprobability{\mu}{D_C}{C} = 1$.
      So,
      \begin{equation*}
        \partitionentropy{\conditionalprobability{\mu}{D_C}{\cdot}}{\family{C}}
        =
        0.
      \end{equation*}
      Therefore,
      using Lemma \ref{lemma:partition_entropy:properties},
      \begin{align*}
        \conditionalpartitionentropy{\mu}{\family{C}}{\family{D}}
        &=
        \mu(D^*)
        \partitionentropy{\conditionalprobability{\mu}{D^*}{\cdot}}{\family{C}}
        \\
        &\leq
        \mu(D^*)
        \log \covercardinality{\family{C}}
        \leq
        1.
      \end{align*}
    }

    \subproof
    {
      $
        \tpartitionentropy{\mu}{\family{D}}{S}
        \leq
        \tcoverentropy{\family{A}}{S}
        +
        \log 2
      $.
    }
    {
      Let us show that
      \begin{equation*}
        \covercardinality{\family{D}^n}
        \leq
        2^n
        \covercardinality{\family{A}^n}.
      \end{equation*}
      Taking the logarithm,
      dividing by $n$ and taking the limit for
      $n \rightarrow \infty$ gives the claim.

      Let
      $\Lambda \subset \set{1, \dotsc, k}^n$
      with
      $\cardinality{\Lambda} = \covercardinality{\family{A}^n}$,
      such that
      \begin{equation*}
        X
        =
        \bigcup_{\lambda \in \Lambda} (K_0 \cup K_{\lambda_0})
        \cap \dotsb \cap
        T^{-(n-1)}(K_0 \cup K_{\lambda_{n-1}}),
      \end{equation*}
      where
      $\lambda = (\lambda_0, \dotsc, \lambda_{n-1})$.
      Consider the mapping
      \begin{equation*}
        \function{f}{\Lambda \times \set{0,1}^n}{\family{D}^n},
      \end{equation*}
      that maps $(\lambda,x)$ to $Y_0 \cap \dotsb \cap Y_{n-1}$,
      where
      \begin{equation*}
        Y_j
        =
        \begin{cases}
          T^{-j} K_{\lambda_j}, &x_j = 1
          \\
          T^{-j} K_0, &x_j = 0.
        \end{cases}
      \end{equation*}
      Since
      \begin{align*}
        X
        &=
        \bigcup_{\lambda \in \Lambda} T^0(K_{\lambda_0} \cup K_0)
        \cap \dotsb \cap T^{-(n-1)}(K_{\lambda_{n-1}} \cup K_0)
        \\
        &=
        \bigcup f(\Lambda \times \{0, 1\}^n),
      \end{align*}
      and $\family{D}^n$ is a partition,
      we have that the image of $f$
      contains every non empty element of $\family{D}^n$.
    }

    From the last claims
    and Lemma \ref{lemma:estimation_using_conditional_entropy},
    we have that
    \begin{align*}
      \measureentropy{\mu}{S}
      &\leq
      \tpartitionentropy{\mu}{\family{C}}{S}
      +
      1
      \\
      &\leq
      \tpartitionentropy{\mu}{\family{D}}{S}
      +
      2
      \\
      &\leq
      \tcoverentropy{\family{A}}{S}
      +
      2
      +
      \log 2
      \\
      &\leq
      \topologicalentropy{S}
      +
      2
      +
      \log 2,
    \end{align*}
    demonstrating
    equation \ref{eq:lemma:principio_variacional:desigualdade_facil:primeira_desigualdade},
    and concluding the proof.
  \end{proof}

  \begin{obs}
    \label{obs:principio_variacional:fortemente_admissivel}
    In Lema \ref{lemma:principio_variacional:desigualdade_facil},
    we have actually shown that
    \begin{equation*}
      \measureentropy{\mu}{T}
      \leq
      \sup_{\family{A} \text{: s-admissible}}
      \tcoverentropy{\family{A}}{T}
      \leq
      \topologicalentropy{T}.
    \end{equation*}
  \end{obs}

  \begin{proof}[Proof of Theorem \ref{th:principio_variacional}]
    \input{principio_variacional_localmente_compacto/02/demonstracao_02}
  \end{proof}



  The inequality at Lemma
  \ref{lemma:sistema_iterado}
  was used to demonstrate the variational principle.
  Now, in its turn, the variational principle allows us
  to go a bit further.

  \begin{corollary}
    \label{corollary:sistema_iterado}
    Consider the metrizable locally compact separable dynamical system
    $\dynamicalsystem{X}{T}$,
    and let $k \in \naturals$.
    Then,
    \begin{equation*}
      \topologicalentropy{T^k}
      =
      k \topologicalentropy{T}.
    \end{equation*}
  \end{corollary}

  \begin{proof}
    Equality holds for the Kolmogorov-Sinai entropy
    (Remark
    \ref{obs:sistema_iterado}).
    Therefore, the variational principle implies that it is
    valid for the topological entropy as well.
  \end{proof}

%% file: principio_variacional_localmente_compacto/02/demonstracao_02.tex
  Using Lemma
  \ref{lemma:principio_variacional:desigualdade_facil},
  we know that
  \begin{equation*}
    \ksentropy{T}
    \leq
    \topologicalentropy{T}
    \leq
    \inf_d
    \bowenentropy{d}{T}
    \leq
    \inf_d
    \dentropy{d}{T},
  \end{equation*}
  where the infimum is taken over every distance $d$
  compatible with the topology in $X$.
  We just have to show that if $r$ is the
  restriction to $X$ of some distance in its one point compactification,
  then,
  \begin{equation*}
    \dentropy{r}{T}
    \leq
    \ksentropy{T}.
  \end{equation*}
  In fact, this implies that
  \begin{equation*}
    \dentropy{r}{T}
    \leq
    \ksentropy{T}
    \leq
    \topologicalentropy{T}
    \leq
    \inf_d
    \bowenentropy{d}{T}
    \leq
    \inf_d
    \dentropy{d}{T}
    \leq
    \dentropy{r}{T}.
  \end{equation*}

  Using Proposition
  \ref{prop:entropia_com_conjuntos_separados},
  it remains to show that,
  for each fixed $\varepsilon > 0$ and each sequence of
  $(n,\varepsilon)$-separated sets $E_n$, there is a Radon measure
  $\mu$, which is $T$-invariant and has total measure lower then or
  equal to $1$, and there is a finite measurable partition $\family{C}$,
  such that for every $n \in \naturals$,
  \begin{equation*}
    \limsup_{n \rightarrow \infty}
    \frac{1}{n}
    \log \cardinality{E_n}
    \leq
    \tpartitionentropy{\mu}{\family{C}}{T}.
  \end{equation*}
  We shall take
  $\dynamicalsystem{Z}{S}$
  as in Lemma \ref{lemma:compactificacao},
  find a probability measure $\mu$ over the Borel sets of $Z$
  and a partition $\family{Z}$ as in
  Lemma \ref{lemma:entropia_na_compactificacao},
  and show that
  \begin{equation*}
    \limsup_{n \rightarrow \infty}
    \frac{1}{n}
    \log \cardinality{E_n}
    \leq
    \tpartitionentropy{\mu}{\family{Z}}{S}.
  \end{equation*}
  Then,
  Lemma \ref{lemma:entropia_na_compactificacao}
  will imply the desired conclusion.

  The demonstration is very similar to that of Theorem 8.6 of \cite{walters},
  the main difference is that we shall use the pseudometric $\widetilde{r}$
  induced by $r$ over $Z$
  (see comment after Lemma \ref{lemma:compactificacao}).
  Let us first build up the Radon probability measure $\mu$,
  and show that it is $S$-invariant.
  Then,
  Lemma \ref{lemma:s_invariant_is_t_invariant}
  implies that $\mu$ is $T$-invariant.
  Define
  \begin{equation*}
    \sigma_n
    =
    \frac{1}{\cardinality{E_n}} \sum_{x \in E_n} \delta_x,
  \end{equation*}
  where $\delta_x$ is the Dirac measure with support in $x$.
  Also define
  \begin{equation*}
    \mu_n = \frac{1}{n} \sum_{j=0}^{n-1} \sigma_n \circ S^{-j},
  \end{equation*}
  and notice that restricted to $X$,
  \begin{equation*}
    \mu_n = \frac{1}{n} \sum_{j=0}^{n-1} \sigma_n \circ T^{-j}.
  \end{equation*}
  In fact, since
  $S^{-j}(Z \setminus X) \subset Z \setminus X$,
  \begin{equation*}
    \mu_n(Z \setminus X)
    =
    \frac{1}{n}
    \sum_{j=0}^{n-1}
    \sigma_n
    \left(
      S^{-j}(Z \setminus X)
    \right)
    =
    0.
  \end{equation*}

  \subproof
  {
    There is a subsequence $n_k$
    and a Radon probability measure $\mu$
    such that $\mu_{n_k} \rightarrow \mu$,
    and such that
    \begin{equation*}
      \lim_{k \rightarrow \infty}
      \frac{1}{n_k}
      \log \cardinality{E_{n_k}}
      =
      \limsup_{n \rightarrow \infty}
      \frac{1}{n}
      \log \cardinality{E_n}.
    \end{equation*}
    Also,
    for any measurable $C \subset Z$
    with $\mu(\partial C) = 0$,
    \begin{equation*}
      \lim
      \mu_{n_k}(C)
      =
      \mu(C).
    \end{equation*}
  }
  {
    In the weak-$*$ topology,
    the set of Radon probability measures $\mu$ over $Z$
    is easily seen to be closed.
    The fact that $Z$ is separable implies that
    the set of Radon measures is separable.
    The Alaoglu Theorem
    (Theorem 2.5.2 of \cite{pedersen})
    implies that it is compact,
    and therefore, sequentially compact.

    It is clear that there is a subsequence $n_k$ such that
    \begin{equation*}
      \lim_{k \rightarrow \infty}
      \frac{1}{n_k}
      \log \cardinality{E_{n_k}}
      =
      \limsup_{n \rightarrow \infty}
      \frac{1}{n}
      \log \cardinality{E_n}.
    \end{equation*}
    From the sequential compacity,
    we can assume that $n_k$ is such that
    $\mu_{n_k}$ converges to some $\mu$.
    The last assertion is a consequence of the
    \emph{Portmanteau Theorem},
    and can be found in \cite{billingsley:convergence},
    Theorem 2.1, item (v).
  }

  \subproof
  {
    The measure $\mu$ is $S$-invariant.
  }
  {
    It is clear that
    $\mu_{n_k} \circ S^{-1} \rightarrow \mu \circ S^{-1}$.
    In fact,
    for any continuous $\function{\phi}{Z}{\reals}$,
    since $R$ is continuous,
    \begin{align*}
      \integral{\phi}{(\mu_{n_k} \circ S^{-1})}
      &=
      \integral{\phi \circ S}{\mu_{n_k}}
      \\
      &\rightarrow
      \integral{\phi \circ S}{\mu}
      \\
      &=
      \integral{\phi}{(\mu \circ S^{-1})}.
    \end{align*}
    On the other hand,
    \begin{align*}
      \abs{\integral{\phi}{(\mu_{n_k} - \mu_{n_k} \circ S^{-1})}}
      &=
      \abs{\integral{\phi \frac{1}{n_k}}{(\sigma_{n_k} - \sigma_{n_k} \circ S^{-n})}}
      \\
      &=
      \frac{1}{n_k}
      \abs{\integral{(\phi - \phi \circ S^{n_k})}{\sigma_{n_k}}}
      \\
      &\leq
      \frac{1}{n_k}
      \integral{\norm[\infty]{\phi - \phi \circ S^{n_k}}}{\sigma_{n_k}}
      \\
      &\leq
      \frac{1}{n_k}
      \integral{2 \norm[\infty]{\phi}}{\sigma_{n_k}}
      \\
      &=
      \frac{1}{n_k}
      2 \norm[\infty]{\phi}
      \rightarrow
      0.
    \end{align*}
    This implies that
    \begin{equation*}
      \mu
      =
      \lim \mu_{n_k}
      =
      \lim \mu_{n_k} \circ S^{-1}
      =
      \mu \circ S^{-1}.
    \end{equation*}
  }

  Now, we choose a measurable partition $\family{Z}$.
  For each $z \in Z$,
  there exists a ball
  $B_z = \ball{\varepsilon_z}{z}$ with
  $\varepsilon_z < \frac{\varepsilon}{2}$,
  such that $\mu(\partial B_z) = 0$.
  Such a $B_z$ exists because
  since the border of such balls are all disjont,
  there is at most a countable number of
  reals $\varepsilon > 0$ such that $\ball{\varepsilon}{z}$
  has border with non null measure.
  Since $Z$ is compact and the balls are open,
  there is a finite number of such balls, $B_0, \dotsc, B_n$ covering $Z$.
  We can assume that $\set{B_0, \dotsc, B_n}$ has no proper sub-cover.
  Let
  \begin{equation*}
    Z_j
    =
    B_j
    \setminus
    \left(
      B_1
      \cup \dotsb \cup
      B_{j-1}
    \right).
  \end{equation*}
  Then,
  $\family{Z} = \set{Z_0, \dotsc, Z_k}$
  is a measurable partition.
  We can also assume that
  $Z \setminus X \subset B_0 = Z_0$.
  That is,
  $\family{Z}$ satisfies the condidtions of
  Lemma \ref{lemma:entropia_na_compactificacao}.

  \subproof
  {
    For each
    $C \in \family{Z}^n$,
    $\mu(\partial C) = 0$.
  }
  {
    Notice that,
    since $S$ is continuous,
    the border operator $\partial$ possesses the following properties.
    \begin{enumerate}
      \item
        \label{it:border:property:complement}
        $\partial A = \partial \complementset{A}$.

      \item
        \label{it:border:property:intersection}
        $
          \partial (A_1 \cap \dotsb \cap A_k)
          \subset
          \partial A_1
          \cup \dotsb \cup
          \partial A_k
        $.

      \item
        \label{it:border:property:continuous}
        $
          \partial S^{-1}(A)
          \subset
          S^{-1}(\partial A)
        $.
    \end{enumerate}
    From items
    \refitem{it:border:property:complement}
    and
    \refitem{it:border:property:intersection},
    each $Z_j = B_j \cap \complementset{B_1} \cap \dotsb \cap \complementset{B_{j-1}}$
    in $\family{Z}$ has border with null measure.
    And from items
    \refitem{it:border:property:intersection}
    and
    \refitem{it:border:property:continuous},
    the same is true for the sets in $\family{Z}^n$.
  }

  Also,
  each $C \in \family{Z}^n$ has diameter less then $\varepsilon$
  in the pseudometric $\iteratedmetric{\widetilde{r}}{n}$,
  just like in the compact case.

  \subproof
  {
    $\log \cardinality{E_n}
    =
    \partitionentropy{\sigma_n}{\family{Z}^n}$.
  }
  {
    Let $C \in \family{Z}^n$.
    If $x,y \in C$, then, there exist
    $C_0, \dotsc, C_{n-1} \in \family{Z}$ such that
    $T^j x, T^j y \in C_j$ for $j = 0, \dotsc, n-1$.  Since each
    element of $\family{Z}$ has diameter less then
    $\varepsilon$, we have that
    $\iteratedmetric{d}{n}(x,y) < \varepsilon$. So, $C$
    can contain at most one element of $E_n$.  That is,
    $\sigma_n(C) = 0$ or
    $\sigma_n(C) = \frac{1}{\cardinality{E_n}}$.
    Therefore,
    \begin{equation*}
      \partitionentropy{\sigma_n}{\family{Z}^n}
      =
      \log \cardinality{E_n}.
    \end{equation*}
  }

  Passing from $\sigma_n$ to $\mu_n$
  is the same procedure as in the compact case,
  as we shall detail right now.
  Notice that for any measurable finite partition $\family{D}$,
  Lemma \ref{lemma:partition_entropy:convex_combination} implies that
  \begin{equation*}
    \sum_{j = 0}^{n-1}
    \frac{1}{n}
    \partitionentropy{\sigma_n \circ S^{-j}}{\family{D}}
    \leq
    \partitionentropy{\mu_n}{\family{D}}.
  \end{equation*}

  For
  $n, q \in \naturals$ with $1 < q < n$,
  take an integer $m$ such that $mq \geq n > m(q-1)$.
  Then, for every $j = 0, \dotsc, q-1$,
  \begin{align*}
    \family{Z}^n
    &\prec
    \family{Z}^j
    \vee
    S^{-j}
    \left(
      \family{Z}^{qm}
    \right)
    \\
    &=
    \family{Z}^j
    \vee
    S^{-j}(\family{Z}^q)
    \vee
    S^{-(j+q)}(\family{Z}^q)
    \vee \dotsb \vee
    S^{-(j+(m-1)q)}(\family{Z}^q).
  \end{align*}
  Therefore,
  using Lemma \ref{lemma:partition_entropy:properties},
  \begin{align*}
    \partitionentropy{\sigma_n}{\family{Z}^n}
    &\leq
    \partitionentropy{\sigma_n}{\family{Z}^j}
    +
    \partitionentropy{\sigma_n \circ S^{-(j + 0q)}}{\family{Z}^q}
    + \dotsb +
    \partitionentropy{\sigma_n \circ S^{-(j + (m-1)q)}}{\family{Z}^q}
    \\
    &\leq
    \partitionentropy{\sigma_n}{\family{Z}^q}
    +
    \partitionentropy{\sigma_n \circ S^{-(j + 0q)}}{\family{Z}^q}
    + \dotsb +
    \partitionentropy{\sigma_n \circ S^{-(j + (m-1)q)}}{\family{Z}^q}
    \\
    &\leq
    \log \cardinality{\family{Z}^q}
    +
    \partitionentropy{\sigma_n \circ S^{-(j + 0q)}}{\family{Z}^q}
    + \dotsb +
    \partitionentropy{\sigma_n \circ S^{-(j + (m-1)q)}}{\family{Z}^q}.
  \end{align*}
  Summing up in $j = 0, \dotsc, q-1$,
  \begin{align*}
    q \log \cardinality{E_n}
    &=
    q \partitionentropy{\sigma_n}{\family{Z}^n}
    \\
    &\leq
    q \log \cardinality{\family{Z}^q}
    +
    \sum_{j=0}^{q-1}
    \sum_{a=0}^{m-1}
    \partitionentropy{\sigma_n \circ S^{-(j + aq)}}{\family{Z}^q}
    \\
    &=
    q \log \cardinality{\family{Z}^q}
    +
    \sum_{p=0}^{n-1}
    \partitionentropy{\sigma_n \circ S^{-p}}{\family{Z}^q}
    +
    \sum_{p=n}^{mq-1}
    \partitionentropy{\sigma_n \circ S^{-p}}{\family{Z}^q}
    \\
    &\leq
    2 q \log \cardinality{\family{Z}^q}
    +
    n
    \sum_{p=0}^{n-1}
    \frac{1}{n}
    \partitionentropy{\sigma_n \circ S^{-p}}{\family{Z}^q}
    \\
    &\leq
    2 q \log \cardinality{\family{Z}^q}
    +
    n \partitionentropy{\mu_n}{\family{Z}^q}.
  \end{align*}
  Since each element $C \in \family{Z}^q$ has border with null measure,
  \begin{equation*}
    \lim_{k \rightarrow \infty}
    \mu_{n_k}(C)
    =
    \mu(C).
  \end{equation*}
  Therefore, dividing by $qn$ and making $k \rightarrow \infty$,
  \begin{align*} 
    \limsup_{n \rightarrow \infty}
    \frac{1}{n}
    \log \cardinality{E_n}
    &=
    \lim_{k \rightarrow \infty}
    \frac{1}{n_k}
    \log \cardinality{E_{n_k}}
    \\
    &\leq
    0
    +
    \frac{1}{q}
    \lim_{k \rightarrow \infty}
    \partitionentropy{\mu_{n_k}}{\family{Z}^q}
    \\
    &=
    \frac{1}{q}
    \partitionentropy{\mu}{\family{Z}^q}
  \end{align*}
  for every $q$.
  Now, one just has to take the limit with $q \rightarrow \infty$
  to get
  \begin{equation*}
    \limsup_{n \rightarrow \infty}
    \frac{1}{n} \log \cardinality{E_n}
    \leq
    \tpartitionentropy{\mu}{\family{Z}}{S}.
  \end{equation*}

%% file: principio_variacional_localmente_compacto/03_endomorfismos_de_grupos_de_lie.tex

\section{Application: Lie group endomorphisms}
  \label{sec:application}

  We finish
  by considering the entropy of continuous endomorphisms of Lie groups.
  We extend some results of \cite{patrao_caldas:endomorfismos} to
  endomorphisms which are not necessarily surjective.
  For a given Lie group $G$, its toral component $T(G)$ is the maximal
  connected and compact subgroup of the center of $G$.
  We aim at demonstrating that in certain cases
  (compact, semisimple, linear reducible and nilpotent),
  the entropy of an endomorphism of a Lie group $G$
  is the entropy of the endomorphism restricted to the toral component $T(G)$.

  In order to reduce the general case to the cases covered in
  \cite{patrao_caldas:endomorfismos},
  a key point is the following lemma,
  which depends on the \emph{variational principle} we have developed.

  \begin{lemma}
    \label{le:restricao_aa_imagem_da_potencia_do_endomorfismo}
    Let $G$ be a connected Lie group and
    $\function{\phi}{G}{G}$
    be a continuous endomorphism.
    There exists a natural number $n$ such that $\phi$ restricted to
    $H = \phi^n(G)$ is surjective.
    In this case,
    \begin{equation*}
      \topologicalentropy{\phi}
      =
      \topologicalentropy{\phi|_H}.
    \end{equation*}
  \end{lemma}

  \begin{proof}
    For the first claim, we first consider the induced Lie algebra
    endomorphism
    $\function{\phi'}{\liealgebra{g}}{\liealgebra{g}}$
    given by the differential of $\phi$ at the identity of $G$.
    Since $\phi'$ is a linear map, there is a natural number $n$ such
    that $(\phi')^n \liealgebra{g} = (\phi')^{n+1} \liealgebra{g}$.
    Putting $H = \phi^n(G)$, we have that its Lie algebra is given by
    $\liealgebra{h} = (\phi')^n \liealgebra{g}$.
    We have that $\phi(H) = H$, since both are connected subgroups and
    the Lie algebra of $\phi(H)$ is
    $\phi' \liealgebra{h} = \liealgebra{h}$.
    Thus $\phi$ restricted to $H$ is surjective.

    For the second claim,
    notice that, as Lie groups,
    $H$ and $G$ satisfy the conditions of
    Theorem \ref{th:principio_variacional}.
    Also,
    since $H$ is a countable union of compact subsets,
    and compact sets of $H$ are also compact sets of $G$,
    $H$ is a measurable subset of $G$.
    Every $\phi$-invariant measure $\mu$ is such that
    \begin{align*}
      \mu(H)
      &=
      \mu
      \left(
        \phi^{-n}(H)
      \right)
      \\
      &=
      \mu(G).
    \end{align*}
    Therefore,
    \begin{align*}
      \topologicalentropy{\phi}
      &=
      \ksentropy{\phi}
      \\
      &=
      \ksentropy{\phi|_H}
      \\
      &=
      \topologicalentropy{\phi|_{H}}.
    \end{align*}
  \end{proof}

  Using the conclusion and notation of
  Lemma \ref{le:restricao_aa_imagem_da_potencia_do_endomorfismo},
  the results in
  \cite{patrao_caldas:endomorfismos}
  imply that in the cases we are considering,
  \begin{equation*}
    \topologicalentropy{\phi}
    =
    \topologicalentropy{\phi|_{H}}
    =
    \topologicalentropy{\phi|_{T(H)}}.
  \end{equation*}
  In order to substitute $T(H)$ by $T(G)$,
  we need to show that in those same cases,
  $T(H) \subset T(G)$.
  This is the content of the next lemma.

  \begin{lemma}
    \label{le:th_tg}
    Let $G$ be a connected reductive or connected nilpotent Lie group
    and $\function{\phi}{G}{G}$
    a continuous endomorphism.
    If $H$ is as in the conclusion of
    Lemma \ref{le:restricao_aa_imagem_da_potencia_do_endomorfismo},
    then $T(H) \subset T(G)$.
  \end{lemma}

  \begin{proof}
    Let $\varphi = \phi^n$ be as in
    Lemma \ref{le:restricao_aa_imagem_da_potencia_do_endomorfismo}.
    Then $\varphi$ induces two surjective homomorphisms
    \begin{align*}
      &
      \functionarray{\varphi_1}{G}{G/T(G)}{x}{\pi(\varphi(x))}
      \\
      &
      \functionarray{\varphi_2}{H}{G/T(G)}{x}{\pi(\varphi(x))},
    \end{align*}
    where $\function{\pi}{G}{G/T(G)}$ is the natural projection.

    Also,
    consider the natural projection
    \begin{equation*}
      \function{\pi_H}{H}{\frac{H}{H \cap T(G)}}.
    \end{equation*}
    Since the image of $\varphi_1$ and $\varphi_2$ are
    $\frac{H T(G)}{T(G)}$,
    \begin{align*}
      G/T(G)
      &\simeq
      \frac{H T(G)}{T(G)}
      \\
      &\simeq
      \frac{H}{H \cap T(G)}
      \\
      &=
      \pi_H(H).
    \end{align*}
    Because $\pi_H$ is surjective,
    it takes the center of $H$ to the center of $\pi_H(H)$.
    In particular,
    $\pi_H(T(H))$ is compact, connected and contained in the center.
    That is,
    \begin{equation}
      \label{eq:th_tg}
      \pi_H(T(H))
      \subset
      T(\pi_H(H)).
    \end{equation}

    We claim that $T(H) \subset T(G)$.
    And this claim follows from
    equation \refeq{eq:th_tg} if we show that
    $T(\pi(H)) = 1$.
    In fact,
    together with equation \refeq{eq:th_tg},
    this implies that
    \begin{equation*}
      T(H)
      \subset
      \ker \pi_H
      =
      H \cap T(G).
    \end{equation*}

    In case $G$ is reductive,
    $G/T(G)$ is semisimple,
    so is $\pi_H(H)$,
    and semisimple groups have trivial toral component.
    In case $G$ is nilpotent,
    Proposition 8 in \cite{patrao_caldas:endomorfismos}
    shows that $G/T(G)$ is simply connected,
    and so is $\pi_H(H)$.
    Since $\pi_H(H)$ is nilpotent and simply connected,
    $Z(\pi_H(H))$ is isomorphic to a finite dimensional vector space,
    and the only compact subgroup of it is the trivial one.
    Therefore,
    the toral component of $\pi_H(H)$ is trivial.
  \end{proof}

  \begin{proposition}
    \label{proposition:compact_group}
    Let $G$ be a connected and compact Lie group and consider
    $\function{\phi}{G}{G}$ a continuous endomorphism.
    Then
    \begin{equation*}
      \topologicalentropy{\phi}
      =
      \topologicalentropy{\phi|_{T(G)}}.
    \end{equation*}
  \end{proposition}

  \begin{proof}
    Consider $H = \phi^n(G)$ given by Lemma
    \ref{le:restricao_aa_imagem_da_potencia_do_endomorfismo}.
    We have that $H$ is a connected and compact Lie group.
    Using Lemma
    \ref{le:restricao_aa_imagem_da_potencia_do_endomorfismo},
    and applying Theorem 6.2 of
    \cite{patrao_caldas:endomorfismos} for $\phi$ restricted to $H$,
    we get that
    \begin{equation*}
      \topologicalentropy{\phi}
      =
      \topologicalentropy{\phi|_{H}}
      =
      \topologicalentropy{\phi|_{T(H)}}.
    \end{equation*}
    Now,
    Lemma \ref{le:th_tg}
    implies that $T(H) \subset T(G)$,
    since a compact Lie group is reductive.
    Therefore,
    we have that
    \begin{equation*}
      \topologicalentropy{\phi}
      =
      \topologicalentropy{\phi|_{T(H)}}
      \leq
      \topologicalentropy{\phi|_{T(G)}}
      \leq
      \topologicalentropy{\phi}.
    \end{equation*}
  \end{proof}

  In the proof of
  Proposition \ref{proposition:compact_group},
  it was evident that $H = \phi^n(G)$ is a compact group.
  It is also evident that $H$ is nilpotent when $G$ is.
  But for the semisimple and reductive cases,
  we need the following lemma.

  \begin{lemma}
    \label{le:imagem_de_grupo}
    Let $G$ be a semisimple,
    or reductive,
    or nilpotent Lie group.
    If
    $\function{\phi}{G}{H}$
    is a surjective Lie group endomorphism,
    then $H$ is respectively
    semisimple,
    reductive
    or nilpotent.
  \end{lemma}

  \begin{proof}
    The nilpotent case is trivial.
    Let $\liealgebra{g}$ denote the Lie algebra of $G$,
    and $\liealgebra{h}$ be the Lie algebra of $H$.

    If $G$ is reductive,
    then
    \begin{equation*}
      \liealgebra{g}
      =
      \liealgebra{a}
      \oplus
      \liealgebra{s},
    \end{equation*}
    where
    $\liealgebra{a}$ is the center of $\liealgebra{g}$,
    and
    \begin{equation*}
      \liealgebra{s}
      =
      \liealgebra{s}_1
      \oplus
      \cdots
      \oplus
      \liealgebra{s}_n,
    \end{equation*}
    where
    $\liealgebra{s}_j$ are simple ideals of $\liealgebra{g}$.

    \subproof
    {
      The Lie algebra
      $\phi'(\liealgebra{a})$
      is in the center of $\liealgebra{h}$.
    }
    {
      Since $\phi'$ is surjective,
      \begin{equation*}
        [\liealgebra{g}, \phi'(\liealgebra{a})]
        =
        [\phi'(\liealgebra{g}), \phi'(\liealgebra{a})]
        =
        \phi'([\liealgebra{g}, \liealgebra{a}])
        =
        0.
      \end{equation*}
    }

    \subproof
    {
      The Lie algebra
      $\phi'(\liealgebra{s})$
      is semisimple.
    }
    {
      Since
      \begin{equation*}
        \phi'(\liealgebra{s})
        =
        \phi'(\liealgebra{s}_1)
        + \dotsb +
        \phi'(\liealgebra{s}_n),
      \end{equation*}
      we just have to show that
      for each $j = 1, \dotsc, n$,
      $\phi'(\liealgebra{s}_j)$
      is simple;
      and for each $j,k = 1, \dotsc, n$,
      either
      $\phi'(\liealgebra{s}_j) = \phi'(\liealgebra{s}_k)$,
      or
      $\phi'(\liealgebra{s}_j) \cap \phi'(\liealgebra{s}_k) = 0$.

      Let $I \subset \phi'(\liealgebra{s}_j)$
      be a non null ideal.
      Then,
      $J = \liealgebra{s}_j \cap \phi'^{-1}(I)$
      is a non null ideal of $\liealgebra{s}_j$.
      Since $\liealgebra{s}_j$ is simple,
      $J = \liealgebra{s}_j$.
      That is,
      \begin{equation*}
        \phi'(\liealgebra{s}_j)
        =
        \phi'(J)
        \subset
        I
        \subset
        \phi'(\liealgebra{s}_j).
      \end{equation*}
      Therefore,
      $\phi'(\liealgebra{s}_j)$ is simple
      whenever it is non null.

      Now,
      let
      $I = \phi'(\liealgebra{s}_j) \cap \phi'(\liealgebra{s}_k)$.
      Then,
      $I$ is an ideal of both,
      $\phi'(\liealgebra{s}_j)$
      and
      $\phi'(\liealgebra{s}_k)$.
      Therefore,
      either
      $I = 0$
      or
      $\phi'(\liealgebra{s}_j) = \phi'(\liealgebra{s}_k)$.
    }

    Denoting by $\liealgebra{h}$ the Lie algebra of $\phi(G)$,
    we know that
    \begin{equation*}
      \liealgebra{h}
      =
      \phi'(\liealgebra{g})
      =
      \phi'(\liealgebra{a})
      +
      \phi'(\liealgebra{s}).
    \end{equation*}
    But,
    since the intersection of an abelian Lie algebra
    and a semisimple Lie algebra is $0$,
    the sum is in fact
    a direct sum.
    And this proves that $\phi(G)$ is reductive.
    If $G$ is semisimple,
    then $\liealgebra{a} = 0$,
    and therefore,
    $\liealgebra{h} = \phi'(\liealgebra{s})$
    is semisimple.
  \end{proof}

  \begin{proposition}
    Let $G$ be a connected semisimple Lie group,
    and let
    $\function{\phi}{G}{G}$ be a continuous endomorphism.
    Then
    \begin{equation*}
      \topologicalentropy{\phi}
      =
      0.
    \end{equation*}
  \end{proposition}

  \begin{proof}
    Consider $H = \phi^n(G)$ given by Lemma
    \ref{le:restricao_aa_imagem_da_potencia_do_endomorfismo}.
    Then,
    Lemma \ref{le:imagem_de_grupo}
    implies that $H$
    is a connected semisimple Lie group,
    and that
    \begin{equation*}
      \topologicalentropy{\phi}
      =
      \topologicalentropy{\phi|_H}
      =
      0,
    \end{equation*}
    where the last equality follows from
    Theorem 5.2 of \cite{patrao_caldas:endomorfismos}.
  \end{proof}

  \begin{proposition}
    \label{proposition:nilpotent:toral_component}
    Let $G$ be a connected linear reductive
    or a connected nilpotent Lie group,
    and let
    $\function{\phi}{G}{G}$ be a continuous endomorphism.
    Then
    \begin{equation*}
      \topologicalentropy{\phi}
      =
      \topologicalentropy{\phi|_{T(G)}}.
    \end{equation*}
    In particular,
    any linear transformation
    $\function{T}{V}{V}$
    over a finite dimensional vector space $V$
    has null topological entropy.
  \end{proposition}

  \begin{proof}
    Consider $H = \phi^n(G)$ given by Lemma
    \ref{le:restricao_aa_imagem_da_potencia_do_endomorfismo}.
    Then,
    Lemma \ref{le:imagem_de_grupo}
    implies that $H$
    is a connected linear reductive
    or
    connected nilpotent.

    Using Lemma
    \ref{le:restricao_aa_imagem_da_potencia_do_endomorfismo},
    and applying Theorem 4.1 (for the nilpotent case) of
    \cite{patrao_caldas:endomorfismos}
    or Corollary 2 (for the linear reductive case) of
    \cite{patrao_caldas:endomorfismos}
    for $\phi$ restricted to $H$,
    we get that
    \begin{equation*}
      \topologicalentropy{\phi}
      =
      \topologicalentropy{\phi|_{H}}
      =
      \topologicalentropy{\phi|_{T(H)}}.
    \end{equation*}
    Now,
    Lemma \ref{le:th_tg},
    implies that $T(H) \subset T(G)$,
    and therefore,
    \begin{equation*}
      \topologicalentropy{\phi}
      =
      \topologicalentropy{\phi|_{T(H)}}
      \leq
      \topologicalentropy{\phi|_{T(G)}}
      \leq
      \topologicalentropy{\phi}.
    \end{equation*}

    The last claim follows from the fact that
    $(V, +)$ is an abelian
    (a \emph{fortiori}, nilpotent) Lie group
    with trivial toral component.
  \end{proof}

%% file: principio_variacional_localmente_compacto.bbl
\providecommand{\bysame}{\leavevmode\hbox to3em{\hrulefill}\thinspace}
\providecommand{\MR}{\relax\ifhmode\unskip\space\fi MR }
\providecommand{\MRhref}[2]{%
  \href{http://www.ams.org/mathscinet-getitem?mr=#1}{#2}
}
\providecommand{\href}[2]{#2}
\begin{thebibliography}{10}

\bibitem{akm:entropia}
R.~Adler, A.~Konheim, and M.~McAndrew, \emph{Topological entropy}, Transactions
  of the American Mathematical Society \textbf{114} (1965), no.~2, 309--319.

\bibitem{billingsley:convergence}
P.~Billingsley, \emph{Convergence of probability measures}, Wiley Series in
  Probability and Statistics, Wiley, 1999.

\bibitem{bowen:entropia}
R.~Bowen, \emph{Entropy for group endomorphisms and homogeneous spaces}, Trans.
  Americ. Math. Soc. \textbf{153} (1971), 401--414.

\bibitem{patrao_caldas:endomorfismos}
A.~Caldas and M.~Patrão, \emph{Dynamics of endomorphisms of {L}ie groups},
  Discrete and Continuous Dynamical Systems \textbf{33} (2013), 1351--1363.

\bibitem{dinaburg}
E.~Dinaburg, \emph{The relation between topological entropy and metric
  entropy}, Soviet Math. Dokl. \textbf{11} (1969), 13--16.

\bibitem{goodman}
T.~Goodman, \emph{Relating topological entropy to measure entropy}, Bull.
  London. Math. Soc. \textbf{3} (1971), 176--180.

\bibitem{handel}
M.~Handel, B.~Kitchens, and D.~Rudolph, \emph{Metrics and entropy for
  non-compact spaces}, Israel Journal of Mathematics \textbf{91} (1995),
  253--271.

\bibitem{misiurewicz}
M.~Misiurewicz, \emph{A short proof of the variational principle for a
  $\integers_+^\naturals$ action on a compact space}, Astérisque \textbf{40}
  (1976), 147--157.

\bibitem{patrao:entropia}
M.~Patrão, \emph{Entropy and its variational principle for non-compact metric
  spaces}, Ergodic Theory and Dynamical Systems \textbf{30} (2010), 1529--1542.

\bibitem{pedersen}
G.~Pedersen, \emph{Analysis now}, Graduate Texts in Mathematics, vol. 118,
  Springer, 1989.

\bibitem{walters}
P.~Walters, \emph{An introduction to ergodic theory}, Springer-Verlag, Berlin,
  2000.

\end{thebibliography}
